\documentclass[final,leqno,onefignum,onetabnum]{siamltex1213}
\usepackage{amsfonts}
\usepackage{enumerate}
\usepackage{cases}

\newtheorem{remark}{{\em Remark}}[section]

\newtheorem{example}{{\em Example}}[section]

\def\mr{\mathbb{R}}
\def\mrn{\mathbb{R}^n}
\def\mrm{\mathbb{R}^m}
\def\mn{\mathbb{N}}
\def\mrd{\mathbb{R}^d}

\def\mmf{\mathbb{F}}
\def\mf{\mathcal{F}}

\def\mb{\mathcal{B}}

\def\ms{\mathbb{S}}

\def\md{\mathbb{D}}
\def\ml{\mathbb{L}}

\def\me{\mathbb{E}}
\def\mmu{\mathcal{U}}

\def\t{\tau}
\def\mx{\mathcal{X}}
\def\dd{\mathcal{D}}
\def\ephs{\varepsilon}

\newcommand{\inner}[2]{\left\langle#1,#2\right\rangle}

\title{Pointwise second-order necessary conditions for stochastic optimal controls, Part I: The case of convex control constraint\thanks{This work is partially supported by the National Basic Research Program
of China (973 Program) under grant 2011CB808002, by NSF of China under grant 11231007, and by the PCSIRT (from the Chinese Education Ministry) under grant IRT1273.}}

\author{Haisen Zhang\thanks{School of Mathematics, Sichuan University, Chengdu 610064, Sichuan Province, China. {\small\it E-mail:} {\small\tt haisenzhang@yeah.net}.}\and
Xu Zhang\thanks{School of Mathematics, Sichuan University, Chengdu 610064, Sichuan Province, China. {\small\it E-mail:} {\small\tt
zhang$\_$xu@scu.edu.cn}.}}

\begin{document}
\maketitle
\slugger{sicon}{xxxx}{xx}{x}{x--x}

\begin{abstract}
This paper is the first part of our series work to establish pointwise second-order necessary conditions for stochastic optimal controls. In this part, both drift and diffusion terms may contain the control variable but the control region is assumed to be convex. Under some assumptions in terms of Malliavin calculus, we establish the desired necessary  condition for stochastic singular optimal controls in the classical sense.
\end{abstract}



\pagestyle{myheadings}
\thispagestyle{plain}
\markboth{H.~Zhang and X.~Zhang}{Pointwise second-order necessary  conditions}

\section{Introduction}
Let $T>0$ and $(\Omega,\mf, \mmf,$ $P)$  be a complete filtered
probability space (satisfying the usual conditions),
on which a $1$-dimensional standard Wiener
process $W(\cdot)$ is defined such that $\mmf=\{\mf_{t} \}_{0\le t\le T}$ is the natural filtration generated by $W(\cdot)$ (augmented by all of the $P$-null sets).

In this paper, we shall consider the following controlled stochastic differential equation
\begin{equation}\label{controlsys}
\left\{
\begin{array}{l}
dx(t)=b(t,x(t),u(t))dt+\sigma(t,x(t),u(t))dW(t),\ \ \ t\in[0,T],\\
x(0)=x_0,
\end{array}\right.
\end{equation}
with a cost functional
\begin{equation}\label{costfunction}
J(u(\cdot))=\me\Big[\int_{0}^{T}f(t,x(t),u(t))dt+h(x(T))\Big].
\end{equation}
Here $u(\cdot)$ is the control variable valued in a set $U\subset \mrm$ (for some $m\in \mn$), $x(\cdot)$ is the state variable valued in $\mr^n$ (for some $n\in \mn$), and $b,\sigma:[0,T]\times \mrn\times U\times\Omega\to \mrn$, $f:[0,T]\times \mrn\times U\times\Omega\to \mr$ and $h:\mrn\times\Omega\to \mr$ are given functions (satisfying some conditions to be given later). As usual, when the context is clear, we omit the $\omega(\in\Omega)$ argument in the defined functions.

Denote by $\mb(\mx)$ the Borel $\sigma$-field of a metric space $\mx$, and by $\mmu_{ad}$ the set of $\mb([0,T])\otimes\mf$-measurable and $\mmf$-adapted stochastic processes valued in $U$. Any $u(\cdot)\in \mmu_{ad}$ is called an admissible control.
The stochastic optimal control problem considered in this paper is to find a control
$\bar{u}(\cdot)\in\mmu_{ad}$  such that
\begin{equation}\label{minimum J}
J(\bar{u}(\cdot))=\inf_{u(\cdot)\in \mmu_{ad}}J(u(\cdot)).
\end{equation}
Any  $\bar{u}(\cdot)\in \mmu_{ad}$ satisfying (\ref{minimum J}) is called an optimal control. The corresponding state  $\bar{x}(\cdot)$ (to (\ref{controlsys})) is called an optimal state, and $(\bar{x}(\cdot),\bar{u}(\cdot))$ is called an optimal pair.

In optimal control theory, one of the central topics is to establish the first-order necessary  condition for optimal controls. We refer to  \cite{Kushner72} for an early study on the first-order necessary condition for stochastic optimal controls. After that, many authors contributed on this topic, see \cite{Bensoussan81, Bismut78, Haussmann76} and references cited therein. Compared to the deterministic setting, new phenomenon and difficulties appear when the diffusion term of the stochastic control system contains the control variable and the control region is nonconvex.  The corresponding first-order necessary condition for this general case was established in \cite{Peng90}.

For some optimal controls, it may happen that the first-order necessary conditions turn out to be trivial. For deterministic control systems, there are two types of such optimal controls. One of them, called the singular optimal control in the classical sense, is the optimal control for which the gradient and the Hessian of the corresponding Hamiltonian with respect to the control variable vanish/degenerate. The other one, called the singular optimal control in the sense of Pontryagin-type maximum principle, is the optimal control for which the corresponding Hamiltonian is equal to a constant in the control region. When an optimal control is singular, the first-order necessary condition cannot provide enough information for the theoretical analysis and numerical computing, and therefore one needs to study the second-order necessary conditions. In the deterministic setting, one can find many references in this direction (See \cite{BellJa75, CA78, Frankowska13, Gabasov72, Goh66, Knobloch81, Krener77, Lou10} and references cited therein).

Compared to the deterministic control systems, there are only two papers (\cite{Bonnans12, Tang10}) addressed to the second-order necessary condition for stochastic optimal controls. In \cite{Tang10}, a pointwise second-order maximum principle for stochastic singular optimal controls in the sense of Pontryagin-type maximum principle was established for the case that the diffusion term $\sigma(t,x,u)$ is independent of the control $u$; while in \cite{Bonnans12}, an integral-type second-order necessary  condition for stochastic optimal controls was derived under the assumption that the control region $U$ is convex.

The main purpose of this paper is to establish a \emph{pointwise} second-order necessary condition for stochastic optimal controls.  In this work, both drift and diffusion terms, i.e., $b(t,x,u)$ and $\sigma(t,x,u)$, may contain the control variable $u$, and we assume that the control region $U$ is convex. The key difference between \cite{Bonnans12} and our work is that we consider here the \emph{pointwise} second-order necessary condition, which is easier to be verified in practical applications. We remark that, quite different from the deterministic setting, there exist some essential difficulties to derive the pointwise second-order necessary condition from an integral-type one when the diffusion term of the control system contains the control variable, \emph{even for the case of convex control constraint} (See the first 4 paragraphs of Subsection \ref{sub3.2} for a detailed explanation). We overcome these difficulties by means of some technique from
the Malliavin calculus. The method developed in this work can be adopted to establish a pointwise second-order necessary condition for stochastic optimal controls for the general case when the control region is nonconvex but the analysis is much more complicated, and therefore we shall give the details in another paper \cite{zhangH14b}.

The rest of this paper is organized as follows. In Section 2, we list some notations, spaces and preliminary results from Malliavin calculus. In Section 3, we introduce the main results of this paper and give some examples. Finally, in Section 4 we give the proofs of the main results.

\section{Some preliminaries}
In this section, we present some preliminaries.
\subsection{Some notations and spaces}

We introduce some notations and spaces which will be used in the sequel.

Denote by $\inner{\cdot}{\cdot}$ and $|\cdot|$ respectively the inner product and norm in $\mrn$ or $\mrm$, which can be identified from the contexts. Let $\mr^{n\times m}$ be the space of all $n\times m$-matrices. For any $A\in \mr^{n\times m}$, denote by $A^{\top}$ the transpose of $A$ and by
$|A|=\sqrt{tr\{AA^{\top}\}}$ the norm of $A$. Also, write $\mathbf{S}^n:=\big\{A\in \mr^{n\times n}\big|\ A^{\top}=A\big\}$.

Let $\varphi: [0,T]\times\mrn\times U\times \Omega\to \mr^{d}$ be a given function. For a.e. $(t,\omega)\in [0,T]\times\Omega$, we denote by $\varphi_{x} (t,x,u)$,  $\varphi_{u} (t,x,u)$ the first order partial derivatives of $\varphi$ with respect to $x$ and $u$ at $(t,x,u,\omega)$, by $\varphi_{(x,u)^2}(t,x,u)$ the Hessian of $\varphi$ with respect to $(x,u)$ at $(t,x,u,\omega)$, and by $\varphi_{xx} (t,x,u)$, $\varphi_{xu} (t,x,u)$, $\varphi_{uu} (t,x,u)$ the second order partial derivatives of $\varphi$ with respect to $x$ and $u$ at $(t,x,u,\omega)$.

For any $\alpha,\beta\in [1,+\infty)$ and $t\in[0,T]$, we denote by $L_{\mf_{t}}^{\beta}(\Omega; \mrn)$ the space of $\mrn$-valued, $\mf_{t}$ measurable random variables $\xi$ such that $\me~|\xi|^{\beta}<+\infty$; by $L^{\beta}([0,T]\times\Omega; \mrn)$ the space of $\mrn$-valued, $\mb([0,T])\otimes \mf$-measurable processes $\varphi$ such that $\|\varphi\|_{\beta}:=\big[\me\int_{0}^{T}|\varphi(t)|^{\beta}dt
\big]^{\frac{1}{\beta}} <+\infty$;
by $L_{\mmf}^{\beta}(\Omega; L^{\alpha}(0,T; \mrn))$ the space of $\mrn$-valued, $\mb([0,T])\otimes \mf$-measurable, $\mmf$-adapted processes $\varphi$ such that $\|\varphi\|_{\alpha,\beta}:=\big[\me~\big(\int_{0}^{T}|\varphi(t)|^{\alpha}dt\big)
^{\frac{\beta}{\alpha}}\big]^{\frac{1}{\beta}} <+\infty$;
by $L_{\mmf}^{\beta}(\Omega; C([0,T]; \mrn))$ the space of $\mrn$-valued, $\mb([0,T])\otimes \mf$-measurable, and $\mmf$-adapted  continuous processes $\varphi$  such that $\|\varphi\|_{\infty,\beta}:=
\big[\me~\big(\sup_{t\in[0,T]}|\varphi(t)|^{\beta}\big)\big]^{\frac{1}{\beta}} <+\infty$, by $L^{\infty}([0,T]\times\Omega; \mrn)$ the space of $\mrn$-valued, $\mb([0,T])\otimes \mf$-measurable processes $\varphi$ such that $\|\varphi\|_{\infty}:=\mbox{ess sup}_{(t,\omega)\in [0,T]\times\Omega}|\varphi(t,\omega)| <+\infty $; and by
$L^{\beta}(0,T;  L_{\mmf}^{\beta}([0,T]\times\Omega; \mrn))$ the $\mrn$-valued, $\mb([0,T])\otimes \mb([0,T])\otimes\mf$ measurable functions $\varphi$ such that for any $t\in[0,T]$, $\varphi(\cdot,t)$ is $\mmf$-adapted and $\|\varphi\|_{\beta}
:=\Big[\me\int_{0}^{T}\int_{0}^{T}|\varphi(s,t)|^{\beta}dsdt
\Big]^{\frac{1}{\beta}}<+\infty$.

\subsection{Some concepts and results from Malliavin calculus}
In this subsection, we recall some concepts and results from Malliavin calculus (See \cite{Nualart06} for a detailed discussion on this topic).

Denote by $C_{b}^{\infty}(\mrd; \mrn)$ the set of $C^{\infty}$-smooth functions with bounded partial derivatives. For any $h\in L^2(0,T)$, write $W(h)=\int_{0}^{T}h(t)dW(t)$. Define
\begin{equation}\label{mJ}
\begin{array}{ll}
\mathcal{S}:=\Big\{\zeta=\varphi(W(h_{1}),\ W(h_{2}),\ \cdots,\  W(h_{d}))\ \Big|\
\varphi\in C_{b}^{\infty}(\mrd; \mrn),\ d\in \mn,\\
 \qquad\qquad\qquad\qquad\qquad\qquad    h_1,h_2,\cdots,h_d\in L^2(0,T) \Big\}.
\end{array}
\end{equation}
Clearly, $\mathcal{S}$ is a linear subspace of $L^{2}_{\mf_{T}}(\Omega; \mrn)$. For any $\zeta\in \mathcal{S}$ (in the form of that in (\ref{mJ})), its Malliavin derivative is defined as follows:
$$\dd_{s}\zeta:=\sum_{i=1}^{d}h_{i}(s)\frac{\partial\varphi}{\partial x_{i}}(W(h_{1}),\ W(h_{2}),\ \cdots,\  W(h_{d})), \ \ \ \ \ s\in[0,T].$$
Write
$$|||\zeta|||_{2}:=\Big[\me ~|\zeta|^2+\me\int_{0}^{T}|\dd_{s}\zeta|^2ds\Big]^{\frac{1}{2}}.$$
Obviously, $|||\cdot|||_{2}$ is a norm on $\mathcal{S}$. It is shown in \cite {Nualart06} that the operator $\dd$ has a closed extension to the space $\md^{1,2}(\mrn)$, the completion of $\mathcal{S}$ with respect to the norm $|||\cdot|||_{2}$. When $\zeta\in \md^{1,2}(\mrn)$, the following Clark-Ocone representation formula holds:
\begin{equation}\label{clark-ocone formula}
\zeta=\me ~\zeta+\int_{0}^{T}\me~(\dd_{s} \zeta\ |\ \mf_{s})dW(s).
\end{equation}
Furthermore, if $\zeta$ is $\mf_{t}$-measurable, then $\dd_{s}\zeta=0$ for any $s\in(t,T]$.

Define $\ml^{1,2}(\mrn)$ to be the space of processes $\varphi\in L^{2}([0,T]\times\Omega; \mrn)$
such that
\begin{enumerate}[{\rm (i)}]
  \item For $a.e.$ $t\in[0,T]$, $\varphi(t,\cdot)\in \md^{1,2}(\mrn)$;
  \item The function $\dd_{s}\varphi(t, \omega):\ [0,T]\times[0,T]\times\Omega\to\mrn $ admits a measurable version; and
  \item $\displaystyle |||\varphi|||_{1,2}:=\Big[\me\int_{0}^{T}|\varphi(t)|^2dt
      +\me\int_{0}^{T}\int_{0}^{T}|\dd_{s}\varphi(t)|^2dsdt\Big]^{\frac{1}{2}}<+\infty.$
\end{enumerate}
Denote by $\ml_{\mmf}^{1,2}(\mrn)$ the set of all adapted processes in $\ml^{1,2}(\mrn)$.

In addition, write
\begin{eqnarray*}
& &\ml_{2^+}^{1,2}(\mrn):=\Big\{\varphi(\cdot)\in\ml^{1,2}(\mrn)\Big|\ \exists\ \dd^{+}\varphi(\cdot)\in L^2([0,T]\times\Omega;\mrn)\ \mbox{such that}\\
& &\qquad\quad f_{\ephs}(s):=\sup_{s<t<(s+\varepsilon)\wedge T}
\me~\big|\dd_{s}\varphi(t)-\dd^{+}\varphi(s)\big|^2<\infty,\ a.e.\ s\in [0,T],\\
& &\quad\qquad f_{\ephs}(\cdot)\ \mbox{is measurable on }\ [0,T]\ \mbox{for any } \varepsilon>0,\  \mbox{and}\ \lim_{\varepsilon\to 0^+}\int_{0}^{T}f_{\ephs}(s)ds=0\Big\};
\end{eqnarray*}
\begin{eqnarray*}
& &\ml_{2^-}^{1,2}(\mrn):=\Big\{\varphi(\cdot)\in\ml^{1,2}(\mrn)\Big|\ \exists\ \dd^{-}\varphi(\cdot)\in L^2([0,T]\times\Omega;\mrn)\ \mbox{such that}\\
& &\qquad\quad g_{\ephs}(s):=\sup_{(s-\varepsilon)\vee 0<t<s}
\me~\big|\dd_{s}\varphi(t)-\dd^{-}\varphi(s)\big|^2<\infty,\ a.e.\ s\in [0,T],\\
& &\qquad\quad g_{\ephs}(\cdot)\ \mbox{is measurable on}\ [0,T]\ \mbox{for any } \varepsilon>0,\ \mbox{and}\ \lim_{\varepsilon\to 0^+}\int_{0}^{T}g_{\ephs}(s)ds=0\Big\}.
\end{eqnarray*}
Denote
$$\ml_{2}^{1,2}(\mrn)=\ml_{2^+}^{1,2}(\mrn)\cap\ml_{2^-}^{1,2}(\mrn).$$
For any $\varphi(\cdot)\in \ml_{2}^{1,2}(\mrn)$,  denote
$\nabla\varphi(\cdot)=\dd^{+}\varphi(\cdot)+\dd^{-}\varphi(\cdot)$.

When $\varphi$ is adapted, $\dd_{s}\varphi(t)=0$ for any $t<s$. In this case, $\dd^{-}\varphi(\cdot)=0$, and $\nabla\varphi(\cdot)=\dd^{+}\varphi(\cdot)$. Denote by $\ml_{2,\mmf}^{1,2}(\mrn)$ the set of all adapted processes in $\ml_{2}^{1,2}(\mrn)$.

Roughly speaking, an element $\varphi\in\ml_{2}^{1,2}(\mrn)$ is a stochastic process whose  Malliavin derivative has suitable continuity on some neighbourhood of $\{(t,t)\ |\  t\in [0,T]\}$. Examples of such process can be found in \cite{Nualart06}. Especially, if $(s,t)\mapsto \dd_{s}\varphi(t)$ is continuous from $V_{\delta}:=\{(s,t)\big|\ |s-t|<\delta,\ s,t\in [0,T]\}$ (for some $\delta>0$) to $L_{\mf_{T}}^{2}(\Omega;\mrn)$, then $\varphi\in\ml_{2}^{1,2}(\mrn)$ and, $\dd^{+}\varphi(t)=\dd^{-}\varphi(t)=\dd_{t}\varphi(t)$.

To end this section, we show the following technical result which will be use in the sequel.

\begin{lemma}\label{lemma for malliavin deriv}
Let $\varphi(\cdot)\in \ml_{2,\mmf}^{1,2}(\mrn)$. Then, there exists a sequence $\{\theta_{n}\}_{n=1}^{\infty}$ of positive numbers such that $\theta_n\to0^+$ as $n\to\infty$ and
\begin{equation}\label{equence converg for malliavin deriv}
\lim_{n\to \infty}\frac{1}{\theta_{n}^2}\int_{\t}^{\t+\theta_{n}}
\int_{\t}^{t}\me~\big|D_{s}\varphi(t)-\nabla \varphi(s)\big|^2dsdt
=0,\quad a.e.\ \t\in[0,T].
\end{equation}
\end{lemma}
\begin{proof}
For any $\tau, \theta\in [0,\infty)$, we take the convention that
$$\sup_{t\in [\t,\t+\theta]\cap[0,T]}
\me\big|D_{\t}\varphi(t)-\nabla\varphi(\t)\big|^2=0$$ whenever $[\t,\t+\theta]\cap[0,T]=\emptyset$.
From the definition of $\ml_{2,\mmf}^{1,2}(\mrm)$, it follows that
\begin{eqnarray*}
& &\lim_{\theta\to 0^+}\frac{1}{\theta^2}\int_{0}^{T}\int_{\t}^{\t+\theta}
\int_{\t}^{t}\me\big|D_{s}\varphi(t)-\nabla \varphi(s)\big|^2dsdtd\t\\
&=&\lim_{\theta\to 0^+}\frac{1}{\theta^2}\int_{0}^{T}\int_{\t}^{\t+\theta}
\int_{s}^{\t+\theta}\me\big|D_{s}\varphi(t)-\nabla\varphi(s)\big|^2dtdsd\t\\
&\le&\lim_{\theta\to 0^+}\frac{1}{\theta}\int_{0}^{T}\int_{\t}^{\t+\theta}
\Big[\sup_{t\in [s,s+\theta]\cap[0,T]}\me
\big|D_{s}\varphi(t)-\nabla\varphi(s)\big|^2\Big]dsd\t\\
&\le&\lim_{\theta\to 0^+}\frac{1}{\theta}\int_{0}^{T}\int_{0}^{\theta}
\Big[\sup_{t\in [s+\t,s+\t+\theta]\cap[0,T]}
\me\big|D_{s+\t}\varphi(t)
-\nabla \varphi(s+\t)\big|^2\Big]dsd\t\\
&\le&\lim_{\theta\to 0^+}\frac{1}{\theta}\int_{0}^{\theta}\int_{0}^{T}
\Big[\sup_{t\in [s+\t,s+\t+\theta]\cap[0,T]}
\me\big|D_{s+\t}\varphi(t)-\nabla\varphi(s+\t)\big|^2\Big]d\t ds \\
&\le&\lim_{\theta\to 0^+}\frac{1}{\theta}\int_{0}^{\theta}\int_{s}^{T}
\Big[\sup_{t\in [\t,\t+\theta]\cap[0,T]}
\me\big|D_{\t}\varphi(t)-\nabla\varphi(\t)\big|^2\Big]d\t ds \\
&\le&\lim_{\theta\to 0^+}\frac{1}{\theta}\int_{0}^{\theta}\int_{0}^{T}
\Big[\sup_{t\in [\t,\t+\theta]\cap[0,T]}
\me\big|D_{\t}\varphi(t)-\nabla\varphi(\t)\big|^2\Big]d\t ds \\
&\le&\lim_{\theta\to 0^+}\int_{0}^{T}
\Big[\sup_{t\in [\t,\t+\theta]\cap[0,T]}
\me\big|D_{\t}\varphi(t)-\nabla\varphi(\t)\big|^2\Big]d\t  \\
&=& 0,
\end{eqnarray*}
which implies (\ref{equence converg for malliavin deriv}).
\end{proof}

\section{Second-order necessary conditions}
In this section, we shall present several second-order necessary conditions for stochastic optimal controls.

To begin with, we assume that

\begin{enumerate}
  \item [{\bf (C1)}] {\em The control region $U $ is nonempty, bounded, and convex.}
  \item [{\bf (C2)}] {\em The functions $b$, $\sigma$, $f$, and $h$ satisfy the following:}
  \begin{enumerate}[{\rm (i)}]
      \item {\em For any $(x, u)\in \mrn\times U$, the stochastic processes
             $b(\cdot, x, u):\ [0,T]\times\Omega\to \mrn$ and $\sigma(\cdot, x, u):\ [0,T]\times\Omega\to \mrn$
             are $\mb([0,T])\otimes\mf$-measurable and $\mmf$-adapted. For a.e. $(t,\omega)\in [0, T]\times\Omega$, the functions
             $b(t, \cdot, \cdot):\ \mrn\times U\to \mrn$ and $\sigma(t, \cdot, \cdot):\ \mrn\times U\to \mrn$
             are continuously differentiable up to order $2$, and all of their partial derivatives are uniformly bounded (with respect to $(t,\omega)\in [0, T]\times\Omega$).
             There exists a constant $L > 0$ such that for a.e. $(t,\omega)\in [0, T]\times\Omega$ and for any  $x,\ \tilde{x}\in\mrn$ and $u,\ \tilde{u}\in U$,}
             $$
             \left\{
             \begin{array}{l}
              |b(t,0, u)|+|\sigma(t,0, u)|\le L,\\
              |b_{(x,u)^2}(t,x,u)-b_{(x,u)^2}(t,\tilde{x},\tilde{u})|\le L(|x-\tilde{x}|+|u-\tilde{u}|),\\
              |\sigma_{(x,u)^2}(t,x,u)-\sigma_{(x,u)^2}(t,\tilde{x},\tilde{u})|\le L(|x-\tilde{x}|+|u-\tilde{u}|).
             \end{array}\right.
             $$

       \item {\em For any $(x, u)\in \mrn\times U$, the stochastic process
             $f(\cdot, x, u):\ [0,T]\times\Omega\to \mr$ is $\mb([0,T])\otimes\mf$-measurable and $\mmf$-adapted, and the random variable $h(x)$ is $\mf_{T}$-measurable. For a.e. $(t,\omega)\in [0, T]\times\Omega$, the functions $f(t, \cdot, \cdot):\ \mrn\times U\to \mr$ and $h(\cdot):\ \mrn \to \mr$ are continuously differentiable up to order $2$, and for any  $x,\ \tilde{x}\in\mrn$ and $u,\ \tilde{u}\in U$,}
             $$
             \left\{
             \begin{array}{l}
              |f(t,x,u)|\le L(1+|x|^{2}+|u|^{2}),\\
              |f_{x}(t,x,u)|+|f_{u}(t,x,u)|\le L(1+|x|+|u|),\\
              |f_{xx}(t,x,u)|+|f_{xu}(t,x,u)|+|f_{uu}(t,x,u)|\le L,\\
              |f_{(x,u)^2}(t,x,u)-f_{(x,u)^2}(t,\tilde{x},\tilde{u})|\le L(|x-\tilde{x}|+|u-\tilde{u}|),\\
              |h(x)|\le L(1+|x|^{2}),\ |h_{x}(x)| \le L(1+|x|),\\
              |h_{xx}(x)| \le L, \ |h_{xx}(x)-h_{xx}(\tilde{x})|\le L|x-\tilde{x}|.
             \end{array}\right.
             $$
  \end{enumerate}
\end{enumerate}

When the condition $(C2)$ is satisfied, the state $x(\cdot)$ (of (\ref{controlsys})) is uniquely defined by any given initial datum $x_{0}\in\mr^n$ and admissible control $u(\cdot)\in \mmu_{ad}$, and the cost functional (\ref{costfunction}) is well-defined on $\mmu_{ad}$. In what follows, $C$ represents a generic constant, depending on $T$ and $L$, but independent of any other parameter, which can be different from line to line.

\subsection{Integral-type second-order conditions}

Let $(\bar{x}(\cdot),\bar{u}(\cdot))$ be an optimal pair, and $u(\cdot)\in \mathcal{U}_{ad}$ be any given admissible control.  Let $\varepsilon\in(0,1)$, and write
\begin{equation}\label{3z1}
v(\cdot)=u(\cdot)-\bar{u}(\cdot),\qquad u^{\varepsilon}(\cdot)=\bar{u}(\cdot)+\varepsilon v(\cdot).
\end{equation}
Since $U$ is convex, $u^{\varepsilon}(\cdot)\in \mmu_{ad}$. Denote by  $x^{\varepsilon}(\cdot)$ the state of (\ref{controlsys}) with respect to the control $u^{\varepsilon}(\cdot)$, and put $\delta x(\cdot)=x^{\varepsilon}(\cdot)-\bar{x}(\cdot)$.
For $\varphi=b,\sigma, f$, denote

\vspace{3mm}

\begin{center}
\setlength{\tabcolsep}{0.5pt}
\begin{tabular*}{13cm}{@{\extracolsep{\fill}}lllr}
& $\varphi_{x}(t)=\varphi_{x}(t,\bar{x}(t),\bar{u}(t))$,
& $\varphi_{u}(t)=\varphi_{u}(t,\bar{x}(t),\bar{u}(t))$,\\
& $\varphi_{xx}(t)=\varphi_{xx}(t,\bar{x}(t),\bar{u}(t))$,
& $\varphi_{xu}(t)=\varphi_{xu}(t,\bar{x}(t),\bar{u}(t))$,\\
& $\varphi_{uu}(t)=\varphi_{uu}(t,\bar{x}(t),\bar{u}(t))$.
& ~~\\
\end{tabular*}
\end{center}

\vspace{2mm}

First, similar to \cite{Bonnans12}, we introduce the following two variational equations:
\begin{equation}\label{firstvariequconvex}
\quad\left\{
\begin{array}{l}
dy_{1}(t)= \Big[b_{x}(t) y_{1}(t)+b_{u}(t)v(t)\Big]dt\\
\qquad\qquad+\Big[\sigma_{x}(t) y_{1}(t)+ \sigma_{u}(t)v(t)\Big]dW(t),\qquad t\in[0,T],\qquad\quad\qquad\quad\\
y_{1}(0)=0
\end{array}\right.
\end{equation}
and
\begin{equation}\label{secondvariequconvex}
\quad\left\{
\begin{array}{l}
dy_{2}(t)= \Big[b_{x}(t)
y_{2}(t)+y_{1}(t)^{\top}b_{xx}(t)y_{1}(t)+2v(t)^{\top}b_{xu}(t)y_{1}(t)\\
\qquad\qquad+v(t)^{\top}b_{uu}(t)v(t)\Big]dt
+\Big[\sigma_{x}(t)y_{2}(t)+y_{1}(t)^{\top}\sigma_{xx}(t)y_{1}(t)\\
\qquad\qquad+2v(t)^{\top}\sigma_{xu}(t)y_{1}(t)
+v(t)^{\top}\sigma_{uu}(t)v(t)\Big]dW(t),\qquad t\in[0,T],\\
y_{2}(0)=0.
\end{array}\right.
\end{equation}

By (\ref{firstvariequconvex})--(\ref{secondvariequconvex}) and similar to \cite[Lemmas 3.5 and 3.11]{Bonnans12}, one has the following estimates.

\vspace{1.5mm}

\begin{proposition}\label{estimateofvariequ}
Let (C2) hold. Then, for any $\kappa\ge 2$,
\begin{eqnarray*}
& & \|y_{1}\|_{\infty,\kappa}^\kappa\le C,
\quad
\|y_{2}\|_{\infty,\kappa}^\kappa\le C,
\quad
\|\delta x\|_{\infty,\kappa}^\kappa\le C\varepsilon^\kappa,
\\
& &
\|\delta x-\varepsilon y_{1}\|_{\infty,\kappa}^\kappa\le C\varepsilon^{2\kappa},
\quad
\|\delta x-\varepsilon y_{1}-\frac{\varepsilon^2}{2}y_{2}\|_{\infty,\kappa}^\kappa\le C\varepsilon^{3\kappa}.
\end{eqnarray*}
\end{proposition}

\begin{proof} The proof is very close to that of \cite[Lemmas 3.5 and 3.11]{Bonnans12}, and therefore, we omit the details.
\end{proof}

\vspace{0.5mm}

Next, define the Hamiltonian
\begin{equation}\label{Hamiltonianconvex}
H(t,x,u, y_{1},z_{1})
:=\inner{y_{1}}{b(t,x,u)}+\inner{z_{1}}{\sigma(t,x,u)}-f(t,x,u),
\end{equation}
$(t,x,u,y_{1},z_{1})\in [0,T]\times\mrn\times U\times\mrn\times\mrn.$
We introduce respectively the following two adjoint equations for (\ref{firstvariequconvex})--(\ref{secondvariequconvex}):
\begin{equation}\label{firstajointequconvex}
 \left\{
\begin{array}{l}
dP_{1}(t)=-\Big[b_{x}(t)^{\top}P_{1}(t)\\
\qquad\qquad \qquad\qquad+\sigma_{x}(t)^{\top}Q_{1}(t)
          -f_{x}(t)\Big]dt+Q_{1}(t)dW(t), \  t\in[0,T], \\
P_{1}(T)=-h_{x}(\bar{x}(T))
\end{array}\right.
\end{equation}
and
\begin{equation}\label{secondajointequconvex}
\quad\left\{
\begin{array}{l}
dP_{2}(t)=-\Big[b_{x}(t)^{\top}P_{2}(t)+P_{2}(t)b_{x}(t) +\sigma_{x}(t)^{\top}P_{2}(t)\sigma_{x}(t) +\sigma_{x}(t)^{\top}Q_{2}(t)\\
\qquad\qquad \qquad
+Q_{2}(t)\sigma_{x}(t)+H_{xx}(t)\Big]dt+Q_{2}(t)dW(t),\  t\in[0,T], \\
P_{2}(T)=-h_{xx}(\bar{x}(T)),
\end{array}\right.
\end{equation}
where $H_{xx}(t)=H_{xx}(t,\bar{x}(t),\bar{u}(t), P_{1}(t),Q_{1}(t))$.

From \cite{Peng97}, it is easy to check that, for any $\beta\ge 1$, the equation (\ref{firstajointequconvex}) admits a unique strong solution $(P_{1}(\cdot),Q_{1}(\cdot))\in L_{\mmf}^{\beta}(\Omega; C([0,T]; \mrn))\times
L_{\mmf}^{\beta}(\Omega; L^{2}(0,T; \mrn))$, and (\ref{secondajointequconvex}) admits a unique strong solution
$(P_{2}(\cdot),Q_{2}(\cdot))\in L_{\mmf}^{\beta}(\Omega; C([0,T]; \mathbf{S}^n))\times
L_{\mmf}^{\beta}(\Omega; L^{2}(0,T; \mathbf{S}^n))$.

Also, we define
\begin{eqnarray}\label{S}
\ms(t,x,u,y_{1},z_{1},y_{2},z_{2})
&:=& H_{xu}(t,x,u,y_{1},z_{1})+b_{u}(t,x,u)^{\top}y_{2}\\
& &+\sigma_{u}(t,x,u)^{\top}z_{2}+\sigma_{u}(t,x,u)^{\top}
y_{2}\sigma_{x}(t,x,u),\nonumber
\end{eqnarray}
$(t,x,u,y_{1},z_{1},y_{2},z_{2})\in [0,T]\times\mrn\times U \times\mrn\times\mrn\times \mathbf{S}^n\times \mathbf{S}^n$,
and denote
\begin{equation}\label{S(t)}
\ms(t)= \ms(t,\bar{x}(t),\bar{u}(t),P_{1}(t),Q_{1}(t),P_{2}(t),Q_{2}(t)) ,\quad t\in [0,T].
\end{equation}

We have the following result.

\begin{proposition}\label{variationalformuforconvex}
Let (C1)--(C2) hold. Then, the following variational equality holds for any $u(\cdot)\in\mmu_{ad}$:
\begin{eqnarray}\label{shorttaylor}
\quad& &J(u^{\varepsilon}(\cdot))-J(\bar{u}(\cdot))\\
&=&-\me\int_{0}^{T}\Big[
\varepsilon \inner{H_{u}(t)}{v(t)}
+\frac{\varepsilon^2}{2} \inner{H_{uu}(t)v(t)}{v(t)}\nonumber\\
& &\quad+\frac{\varepsilon^2}{2}
\inner{P_{2}(t)\sigma_{u}(t)v(t)}{\sigma_{u}(t)v(t)}
+\varepsilon^2 \inner{\ms(t)y_{1}(t)}{v(t)}\Big]dt
+o(\varepsilon^{2}), \  (\ephs\to 0^+),\nonumber
\end{eqnarray}
where $H_{u}(t)\!=\!H_{u}(t,\bar{x}(t),\bar{u}(t), P_{1}(t),Q_{1}(t))$,
$H_{uu}(t)\!=\!H_{uu}(t,\bar{x}(t),\bar{u}(t), P_{1}(t),Q_{1}(t))$.
\end{proposition}
\begin{proof}
By (\ref{3z1}), using Taylor's formula and Proposition \ref{estimateofvariequ}, similar to \cite[Subsection 3.2]{Bonnans12}, we have
\begin{eqnarray}\label{taylorexpconvex}
& & J(u^{\varepsilon})-J(\bar{u})\\
&=& \me\int_{0}^{T}\Big[\inner{f_x(t)}{\delta x(t)}+ \varepsilon \inner{f_{u}(t)}{v(t)}+\frac{1}{2}\inner{f_{xx}(t)\delta x(t)}{\delta x(t)}
\nonumber\\
& &
+\varepsilon\inner{f_{xu}(t)\delta x(t)}{v(t)}+\frac{\varepsilon^2}{2}\inner{f_{uu}(t)v(t)}{v(t)}
\Big]dt\nonumber\\
& &+\me \Big[\inner{h_{x}(\bar{x}(T))}{\delta x(T)}+\frac{1}{2}\inner{h_{xx}(\bar{x}(T))\delta x(T)}{\delta x(T)}
\Big]+o(\varepsilon^2)\quad (\ephs\to 0^+)\nonumber\\
&=& \me\int_{0}^{T}\Big[\varepsilon\inner{f_{x}(t)}{ y_{1}(t)}
+\frac{\varepsilon^2}{2}\inner{f_{x}(t)}{y_{2}(t)}
+\varepsilon\inner{f_{u}(t)}{v(t)}\nonumber\\
& &+\frac{\varepsilon^2}{2}\Big(\inner{f_{xx}(t)y_{1}(t)}{y_{1}(t)}
+2\inner{f_{xu}(t)y_{1}(t)}{v(t)}
+\inner{f_{uu}(t)v(t)}{v(t)}\Big)\Big]dt\nonumber\\
& &
+\me \Big[\varepsilon \inner{h_{x}(\bar{x}(T))}{y_{1}(T)}
+\frac{\varepsilon^2}{2} \inner{h_{x}(\bar{x}(T))}{y_{2}(T)} \nonumber\\
& &+\frac{\varepsilon^2}{2}\inner{h_{xx}(\bar{x}(T))y_{1}(T)}{y_{1}(T)}\Big]
+ o(\varepsilon^2),\quad (\ephs\to 0^+).\nonumber
\end{eqnarray}

By It\^{o}'s formula, we have
\begin{eqnarray}\label{hxty1}
& &\me ~\inner{h_{x}(\bar{x}(T))}{y_{1}(T)}
=-\me ~\inner{P_{1}(T)}{y_{1}(T)}\\
&=&-\me\int_{0}^{T}\Big[ \inner{P_{1}(t)}{b_{u}(t)v(t)}
+\inner{Q_{1}(t)}{\sigma_{u}(t)v(t)}+\inner{f_{x}(t)}{y_{1}(t)}
\Big]dt,\nonumber
\end{eqnarray}
\begin{eqnarray}\label{hxty2}
& &\me ~\inner{h_{x}(\bar{x}(T))}{y_{2}(T)}
= -\me ~\inner{P_{1}(T)}{y_{2}(T)}\\
&=&-\me\int_{0}^{T}\Big[ \inner{P_{1}(t)}{y_{1}(t)^{\top}b_{xx}(t)y_{1}(t)}
+2\inner{P_{1}(t)}{v(t)^{\top}b_{xu}(t)y_{1}(t)}\nonumber\\
& &\quad\quad\quad\quad
+\inner{P_{1}(t)}{v(t)^{\top}b_{uu}(t)v(t)}
+\inner{Q_{1}(t)}{y_{1}(t)^{\top}\sigma_{xx}(t)y_{1}(t)}\nonumber\\
& &\quad\quad\quad\quad
+2\inner{Q_{1}(t)}{v(t)^{\top}\sigma_{xu}(t)y_{1}(t)}
+\inner{Q_{1}(t)}{v(t)^{\top}\sigma_{uu}(t)v(t)}\nonumber\\
& &\quad\quad\quad\quad
+\inner{f_{x}(t)}{y_{2}(t)}\Big]dt,\nonumber
\end{eqnarray}
and (noting that $P_2(t)^{\top}=P_2(t)$ and $Q_2(t)^{\top}=Q_2(t)$)
\begin{eqnarray}\label{hxxty12}
& &\me ~\inner{ h_{xx}(\bar{x}(T))y_{1}(T)}{y_{1}(T)}
= -\me ~\inner{P_{2}(T)y_{1}(T)}{y_{1}(T)}\\
&=&-\me\int_{0}^{T}\Big[\inner{P_{2}(t)y_{1}(t)}{b_{u}(t)v(t)}
+\inner{P_{2}(t)b_{u}(t)v(t)}{y_{1}(t)}\nonumber\\
& &\qquad\quad\quad
+\inner{P_{2}(t)\sigma_{x}(t)y_{1}(t)}{\sigma_{u}(t)v(t)}
+\inner{P_{2}(t)\sigma_{u}(t)v(t)}{\sigma_{x}(t)y_{1}(t)}\nonumber\\
& &\qquad\quad\quad
+\inner{P_{2}(t)\sigma_{u}(t)v(t)}{\sigma_{u}(t)v(t)}
+\inner{Q_{2}(t)\sigma_{u}(t)v(s)}{y_{1}(t)}\nonumber\\
& &\qquad\quad\quad
+\inner{Q_{2}(t)y_{1}(t)}{\sigma_{u}(t)v(t)}
-\inner{H_{xx}(t)y_{1}(t)}{y_{1}(t)}
\Big]dt\nonumber\\&=&-\me\int_{0}^{T}\Big[2\inner{P_{2}(t)y_{1}(t)}{b_{u}(t)v(t)}
+2\inner{P_{2}(t)\sigma_{x}(t)y_{1}(t)}{\sigma_{u}(t)v(t)}
\nonumber\\
& &\qquad\quad\quad
+\inner{P_{2}(t)\sigma_{u}(t)v(t)}{\sigma_{u}(t)v(t)}
+2\inner{Q_{2}(t)\sigma_{u}(t)v(s)}{y_{1}(t)}\nonumber\\
& &\qquad\quad\quad
-\inner{H_{xx}(t)y_{1}(t)}{y_{1}(t)}
\Big]dt.\nonumber
\end{eqnarray}

Substituting (\ref{hxty1}), (\ref{hxty2}) and (\ref{hxxty12}) into (\ref{taylorexpconvex}), we obtain that
\begin{eqnarray*}
& &J(u^{\varepsilon})-J(\bar{u})\nonumber\\
&=&-\me\int_{0}^{T}\Big[\varepsilon\Big(
\inner{P_{1}(t)}{b_{u}(t)v(t)}
+\inner{Q_{1}(t)}{\sigma_{u}(t)v(t)}
-\inner{f_{u}(t)}{v(t)}\Big)\nonumber\\
& &\qquad\quad
+\frac{\varepsilon^2}{2}\Big(\inner{P_{1}(t)}{v(t)^{\top}b_{uu}(t)v(t)}
+\inner{Q_{1}(t)}{v(t)^{\top}\sigma_{uu}(t)v(t)}\nonumber\\
& &\qquad\quad
-\inner{f_{uu}(t)v(t)}{v(t)}\Big)
+\frac{\varepsilon^2}{2}\inner{P_{2}(t)
\sigma_{u}(t)v(t)}{\sigma_{u}(t)v(t)}\nonumber\\
& &\qquad\quad
+\varepsilon^2\Big(
\inner{P_{1}(t)}{v(t)^{\top}b_{xu}(t)y_{1}(t)}
+\inner{Q_{1}(t)}{v(t)^{\top}\sigma_{xu}(t)y_{1}(t)}\nonumber\\
& &\qquad\quad
-\inner{f_{xu}(t)y_{1}(t)}{v(t)}
+\inner{b_{u}(t)^{\top}P_{2}(t)y_{1}(t)}{v(t)}\nonumber\\
& &\qquad\quad
+\inner{\sigma_{u}(t)^{\top}P_{2}(t)\sigma_{x}(t)y_{1}(t)}{v(t)}
\nonumber\\
& &\qquad\quad
+\inner{\sigma_{u}(t)^{\top}Q_{2}(t)y_{1}(t)}{v(t)}\Big)
\Big]dt+o(\varepsilon^{2})\qquad (\ephs\to 0^+)\nonumber\\
&=&-\me\int_{0}^{T}\Big[
\varepsilon\inner{H_{u}(t)}{v(t)}
+\frac{\varepsilon^2}{2}\inner{H_{uu}(t)v(t)}{v(t)}\nonumber\\
& &\qquad\quad+\frac{\varepsilon^2}{2} \inner{
P_{2}(t)\sigma_{u}(t)v(t)}{\sigma_{u}(t)v(t)}\\
& &\qquad\quad+\varepsilon^2\inner{\ms(t)y_{1}(t)}{v(t)}
\Big]dt+o(\varepsilon^{2}),\quad (\ephs\to 0^+).
\end{eqnarray*}
This completes the proof of Proposition \ref{variationalformuforconvex}.
\end{proof}

Now, we establish an integral-type second-order necessary condition for
stochastic optimal controls. Stimulated by \cite{Gabasov72}, we introduce the following notion.
\begin{definition}\label{singularcontrol calssical}
We call a control $\tilde{u}(\cdot)\in \mmu_{ad}$ a singular control in the classical sense if $\tilde{u}(\cdot)$ satisfies
\begin{equation}\label{singularcontrol Hu Huu euql zero}
\quad\ \ \ \ \ \ \  \left\{\!\!\!
\begin{array}{l}
H_{u}(t,\tilde{x}(t), \tilde{u}(t) ,\tilde{P}_{1}(t),\tilde{Q}_{1}(t))=0,\ \  \ a.s., \ a.e.\ t\in [0,T], \\
H_{uu}(t,\tilde{x}(t), \tilde{u}(t) ,\tilde{P}_{1}(t),\tilde{Q}_{1}(t))+\sigma_{u}(t,\tilde{x}(t),\tilde{u}(t))^{\top}
\tilde{P}_{2}(t)\sigma_{u}(t,\tilde{x}(t),\tilde{u}(t))=0,\\
 \ \ \ \ \ a.s., \ a.e.\ t\in [0,T],
\end{array}\right.
\end{equation}
where $\tilde{x}(\cdot)$ is the state with respect to $\tilde{u}(\cdot)$, and $(\tilde{P}_{1}(\cdot),\tilde{Q}_{1}(\cdot))$ and $(\tilde{P}_{2}(\cdot),\tilde{Q}_{2}(\cdot))$ be the adjoint processes given respectively by (\ref{firstajointequconvex}) and (\ref{secondajointequconvex}) with $(\bar x(\cdot),\bar u(\cdot)
)$ replaced by $(\tilde{x}(\cdot),\tilde{u}(\cdot))$.
\end{definition}

\begin{remark}
Since the diffusion term $\sigma(t,x,u)$ contains the control variable $u$, in order to represent the stochastic maximum principle, one needs to introduce the following $\cal H$-function:
  \begin{eqnarray*}
 {\cal H}(t,x,u)
&:=&H(t,x,u, \tilde{P}_{1}(t),\tilde{Q}_{1}(t))-\frac{1}{2}\inner{\tilde P_{2}(t)\sigma(t,\tilde{x}(t),\tilde{u}(t))}
{\sigma(t,\tilde{x}(t),\tilde{u}(t))}\\
& &
+\frac{1}{2}\big\langle\tilde P_{2}(t)\big(\sigma(t,x,u)
-\sigma(t,\tilde{x}(t),\tilde{u}(t))\big),\sigma(t,x,u)-\sigma(t,\tilde{x}(t),\tilde{u}(t))\big\rangle.
\end{eqnarray*}
The stochastic maximum principle (see \cite{Peng90}) says, if $(\tilde{x}(\cdot),\tilde{u}(\cdot))$ is an optimal pair, then
\begin{equation}\label{maximum principle}
{\cal H}(t,\tilde{x}(t),\tilde{u}(t))
=\max_{v\in U} {\cal H}(t,\tilde{x}(t),v),\quad a.s., \ a.e.\ t\in [0,T].
\end{equation}
A singular control in the classical sense is the one that satisfies trivially the first- and second-order necessary conditions in optimization theory for the maximization problem (\ref{maximum principle}), i.e.,
 \begin{equation}\label{singularcontreuql zero}
\quad\ \ \  \left\{\!\!\!
\begin{array}{l}
{\cal H}_{u}(t,\tilde{x}(t), \tilde{u}(t))=0,\ \  \ a.s., \ a.e.\ t\in [0,T], \\
{\cal H}_{uu}(t,\tilde{x}(t), \tilde{u}(t)=0, \ \ \ a.s., \ a.e.\ t\in [0,T].
\end{array}\right.
 \end{equation}
It is easy to see that (\ref{singularcontreuql zero}) is equivalent to (\ref{singularcontrol Hu Huu euql zero}).
On the other hand, one can consider (stochastic) singular optimal controls in other senses, say in the sense of
Pontryagin-type maximum principle. Due to the space limitation, we shall present our results in this respect elsewhere.
\end{remark}

By Proposition \ref{variationalformuforconvex}, we obtain the following integral-type second-order necessary  condition.

\begin{theorem}
Let (C1)--(C2) hold. If $\bar{u}(\cdot)$ is a singular optimal control in the classical sense, then
\begin{equation}\label{integraltype 2ordercondition}
\me\int^{T}_{0}\inner{\ms(t)y_{1}(t)}{v(t)}dt\le 0,
\end{equation}
for any $v(\cdot)=u(\cdot)-\bar{u}(\cdot)$, $u(\cdot)\in \mmu_{ad}$.
\end{theorem}
\begin{proof}
By (\ref{shorttaylor}) and Definition \ref{singularcontrol calssical}, we have
\begin{equation}\nonumber
0\le\lim_{\varepsilon\to 0^+}\frac{J(u^{\varepsilon})-J(\bar{u})}{\varepsilon^2}
  =-\me\int_{0}^{T}\inner{\ms(t)y_{1}(t)}{v(t)}dt,
\end{equation}
as stated.
\end{proof}

In \cite{Bonnans12}, the authors obtained the following integral-type first- and second-order necessary conditions for stochastic optimal controls:
\begin{theorem}
Let (C1)--(C2) hold.
If $\bar{u}(\cdot)$ is an optimal control, then
$$\int^{T}_{0}\inner{H_{u}(t)}{w(t)}dt\le0,\qquad \forall
\ w(\cdot)\in cl_{2,2}\big(\mathcal{R}_{\mmu_{ad}}(\bar{u})\cap L^{4}_{\mmf}(\Omega;L^{4}(0,T;\mrm))\big).$$
Furthermore, for any  $w(\cdot)\in cl_{4,4}\big(\mathcal{R}_{\mmu_{ad}}(\bar{u})\cap L^{\infty}([0,T]\times \Omega;\mrm)\cap \Upsilon(\bar{u})\big)$ the following second-order necessary condition holds:
\begin{eqnarray}\label{bonnans 2order condition}
& &\me\int^{T}_{0}\Big[\inner{H_{xx}(t)y_{1}(t)}{y_{1}(t)}
+2\inner{H_{xu}(t)y_{1}(t)}{w(t)}\nonumber\\
& &\quad\quad +\inner{H_{uu}(t)w(t)}{w(t)}\Big]dt
 +\me\inner{h_{xx}(\bar{x}(T))y_{1}(T)}{y_{1}(T)}\le0.~~\qquad
\end{eqnarray}
Here,
$$\mathcal{R}_{\mmu_{ad}}(\bar{u})
:=\Big\{\alpha u(\cdot)-\alpha\bar{u}(\cdot)\ \Big|\
u(\cdot)\in \mmu_{ad}, \alpha>0\Big\},$$
$$\Upsilon(\bar{u}):=\Big\{w(\cdot)\in L^{2}_{\mmf}(\Omega;L^{2}(0,T;\mrm))\ \Big|\ \int^{T}_{0}\inner{H_{u}(t)}{w(t)}dt=0\Big\}$$
and, $cl_{2,2}(A)$ and $cl_{4,4}(A)$ are the closure of a set $A$ under the norms $\|\cdot\|_{2,2}$ and $\|\cdot\|_{4,4}$, respectively.
\end{theorem}

There are some second-order terms with respect to $y_{1}(\cdot)$  in
(\ref{bonnans 2order condition}). These terms are eliminated in
(\ref{integraltype 2ordercondition}) by introducing the second-order adjoint process $(P_{2}(\cdot),Q_{2}(\cdot))$.
Note also that, the second-order necessary condition we consider in this paper is for the singular optimal controls in the classical sense, hence the second order terms $\inner{H_{uu}(t)v(t)}{v(t)}$  and $\inner{P_{2}(t)\sigma_{u}(t)v(t)}{\sigma_{u}(t)v(t)}$ appearing in the variational formulation (\ref{shorttaylor}) do not enter into
(\ref{integraltype 2ordercondition}).

\subsection{Second-order necessary condition in term of martingale representation}\label{sub3.2} Let us recall that, in order to derive pointwise necessary conditions for optimal controls, one needs to establish first some suitable integral-type necessary conditions. It is well-known that there is no difficulty to establish the pointwise first-order necessary condition for optimal controls whenever an integral-type one is obtained. However, the classical method of deriving the pointwise condition from the integral-type one cannot be used directly to establish the pointwise second-order condition in the general stochastic setting.

Note that the solution $y_{1}(\cdot)$ to the first variational equation (\ref{firstvariequconvex}) appears in the integral-type second-order condition (\ref{integraltype 2ordercondition}). By \cite[Theorem 1.6.14, p.47]{Yong99}, $y_{1}(\cdot)$ enjoys an explicit representation:
\begin{eqnarray}\label{y1(t)for v(t)}
 y_{1}(t)&=&\Phi(t)\int_{0}^{t}\Phi(s)^{-1}
\big(b_{u}(s)-\sigma_{x}(s)\sigma_{u}(s)
\big)v(s)ds\nonumber\\
& &+\Phi(t)\int_{0}^{t}\Phi(s)^{-1}\sigma_{u}(s)v(s)dW(s),
\end{eqnarray}
where
$\Phi(\cdot)$ is the solution to the following matrix-valued stochastic differential equation
\begin{equation}\label{Phi}
\left\{
\begin{array}{l}
d\Phi(t)= b_{x}(t)\Phi(t)dt+\sigma_{x}(t)\Phi(t)dW(t),
\qquad \ \ \ t\in[0,T], \qquad \\
\Phi(0)=I,
\end{array}\right.
\end{equation}
and $I$ stands for the identity matrix in $\mr^{n\times n}$. Substituting the explicit representation (\ref{y1(t)for v(t)}) of $y_{1}(\cdot)$ into (\ref{integraltype 2ordercondition}), we see that there will appear a ``bad" term of the following form:
\begin{equation}\label{double integral ito and lebesgue}
\me\int_{0}^{T}
\Big\langle\ms(t)\Phi(t)\int_{0}^{t}
\Phi(s)^{-1}\sigma_{u}(s)
v(s)dW(s),
v(t)\Big\rangle dt.
\end{equation}

To see (\ref{double integral ito and lebesgue}) is ``bad", let us choose $\t\in [0,T)$, $v\in U$, $E_{\theta}=[\t,\t+\theta)$ such that $\theta>0$ and $\t+\theta\le T$. Denote by $\chi_{E_{\theta}}(\cdot)$ the characteristic function of the set $E_{\theta}$. As usual, though the control region $U$ is convex, in order to derive a pointwise second-order necessary condition from the integral one (\ref{integraltype 2ordercondition}), people need to choose the following needle variation for the optimal control $\bar{u}(\cdot)$:
\begin{equation}\label{needle variation}
u(t)=\left\{
\begin{array}{l}
v, \qquad\qquad t\in E_{\theta},\\
\bar{u}(t), \qquad \quad t\in [0,T] \setminus E_{\theta}.\\
\end{array}\right.
\end{equation}
For this $u(\cdot)$, it is clear that $v(\cdot)=u(\cdot)-\bar{u}(\cdot)=(v-\bar{u}(\cdot))\chi_{E_{\theta}}(\cdot)$, and (\ref{double integral ito and lebesgue}) is reduced to
\begin{equation}\label{3 per 2 order term}
\me\int_{\t}^{\t+\theta}
\Big\langle\ms(t)\Phi(t)\int_{\t}^{t}
\Phi(s)^{-1}\sigma_{u}(s)
(v-\bar{u}(s))dW(s),
v-\bar{u}(t)\Big\rangle dt.
\end{equation}
Since an It\^{o} integral appears in (\ref{3 per 2 order term}), we have
\begin{eqnarray*}
& &\me\int_{\t}^{\t+\theta}
\Big\langle\ms(t)\Phi(t)\int_{\t}^{t}
\Phi(s)^{-1}\sigma_{u}(s)
(v-\bar{u}(s))dW(s),
v-\bar{u}(t)\Big\rangle dt \\
&\le & \!\Big[\me\int_{\t}^{\t+\theta}\Big|\big(\ms(t)\Phi(t) \big)^{\top}\big( v-\bar{u}(t)\big)
\Big|^{2} dt \!\Big]^{\frac{1}{2}}
\Big[\me\int_{\t}^{\t+\theta}\int_{\t}^{t}
\Big|\Phi(s)^{-1}\sigma_{u}(s)
(v-\bar{u}(s))\Big|^{2}ds dt\! \Big]^{\frac{1}{2}}\\
&=&O(\theta^{\frac{3}{2}}),\quad (\theta\to 0^+).
\end{eqnarray*}
Because of this, it seems that (\ref{3 per 2 order term}) is not an infinitesimal of order $2$ but only that of order $\frac{3}{2}$ with respect to $\theta$ (as $\theta\to 0^+$).

However, by the properties of It\^o's integral, we find that
\begin{eqnarray*}
& &\lim_{\theta\to 0^+}\Big|\frac{1}{\theta^{\frac{3}{2}}}\me\int_{\t}^{\t+\theta}
\Big\langle\ms(t)\Phi(t)\int_{\t}^{t}
\Phi(s)^{-1}\sigma_{u}(s)
(v-\bar{u}(s))dW(s),
v-\bar{u}(t)\Big\rangle dt\Big|\\
&\le&\lim_{\theta\to 0^+}\Big|\frac{1}{\theta^{\frac{3}{2}}}\me\int_{\t}^{\t+\theta}
\Big\langle\big(\ms(t)\Phi(t) \big)^{\top}\big( v-\bar{u}(t)\big)-\big(\ms(\t)\Phi(\t) \big)^{\top}\big( v-\bar{u}(\t)\big), \\
& &\qquad\qquad\qquad\qquad\qquad\qquad\qquad\qquad\qquad\quad\
\int_{\t}^{t}\Phi(s)^{-1}\sigma_{u}(s)(v-\bar{u}(s))dW(s)\Big\rangle dt\Big|\\
& & +\lim_{\theta\to 0^+}\Big|\frac{1}{\theta^{\frac{3}{2}}}\me\int_{\t}^{\t+\theta}
\Big\langle\big(\ms(\t)\Phi(\t) \big)^{\top}\big( v-\bar{u}(\t)\big), \int_{\t}^{t}
\Phi(s)^{-1}\sigma_{u}(s)
(v-\bar{u}(s))dW(s)\Big\rangle dt\Big|\\
&\le & \lim_{\theta\to 0^+}\frac{1}{\theta^{\frac{3}{2}}}
\Big[\me\int_{\t}^{\t+\theta}\Big|\big(\ms(t)\Phi(t) \big)^{\top}\big( v-\bar{u}(t)\big)-\big(\ms(\t)\Phi(\t) \big)^{\top}\big( v-\bar{u}(\t)\big)
\Big|^{2} dt \Big]^{\frac{1}{2}}\cdot\\
& &\qquad\qquad\qquad\qquad\qquad\qquad\qquad\qquad
\Big[\me\int_{\t}^{\t+\theta}\int_{\t}^{t}
\Big|\Phi(s)^{-1}\sigma_{u}(s)
(v-\bar{u}(s))\Big|^{2}ds dt \Big]^{\frac{1}{2}}\\
&=&0, \qquad a.e. \;\t\in[0,T).
\end{eqnarray*}
This indicates that, (3.22) is actually a higher order infinitesimal of $\theta^{\frac{3}{2}}$ (as $\theta\to 0^+$).

Essentially, the above problem is caused by the It\^{o} integral.
Indeed, one cannot use the Lebesgue differentiation theorem directly to treat the It\^{o} integral appeared in (\ref{3 per 2 order term}). In this subsection, we shall reduce the It\^{o}-Lebesgue integral term (\ref{3 per 2 order term}) to a double Lebesgue integral term by means of the property of It\^{o}'s integrals and the martingale representation theorem, and obtain a second-order necessary condition for singular optimal controls.

We need the following technical result (which should be known but we do not find an exact reference).

\begin{lemma}\label{martingale exp for L2 function}
Let $\varphi(\cdot)\in L^{2}_{\mmf}(\Omega;L^{2}(0,T;\mrn))$. Then, there exists a $\phi(\cdot,\cdot)\in L^{2}(0,T; $ $ L_{\mmf}^{2}([0,T]\times\Omega; \mrn))$ such that
\begin{equation}\label{martingale exp formulation}
\varphi(t)=\me\,\varphi(t)+\int_{0}^{t}\phi(s,t)dW(s), \quad a.s., \ a.e. \ t\in[0,T].
\end{equation}
\end{lemma}
\begin{proof}
Let $\{\varphi_{j}(\cdot)\}_{j=1}^{\infty}$
be a sequence in $L^{2}_{\mmf}(\Omega;L^{2}(0,T;\mrn))$ such that
$$\me~\int_{0}^{T}\big|\varphi_{j}(t)-\varphi(t)\big|^2dt\to 0,\qquad\hbox{as } j\to \infty,$$
where $\varphi_{j}(\cdot)=\sum_{k=0}^{K_{j}}\xi_{j}^{k}\chi_{[t_{k},t_{k+1})}(t)$, $K_{j}\in \mn$, $0=t_{0}< t_{1}<\cdots<t_{K_{j}+1}=T$ is a partition of $[0,T]$, and $\xi_{j}^{k}\in L^{2}_{\mf_{t_k}}(\Omega;\mrn)$.

For any fixed $j$ and $k$, since $\xi_{j}^{k}\in L^{2}_{\mf_{t_k}}(\Omega;\mrn)$, by the martingale representation theorem, there exists a stochastic process $\phi_{j}^{k}(\cdot)\in L^{2}_{\mmf}(\Omega;L^{2}(0,T;$ $\mrn))$ such that
$$\xi_{j}^{k}=\me~\xi_{j}^{k}+\int_{0}^{t_k}\phi_{j}^{k}(s)dW(s),\quad a.s.$$
Define
$$\phi_{j}(s,t)=\sum_{k=0}^{K_{j}}\phi_{j}^{k}(s)\chi_{[0,t_k]}(s)\chi_{[t_{k},t_{k+1})}(t),
\quad (s,t)\in [0,T]\times [0,T].$$
Clearly, $\varphi_{j}(\cdot)$ can be represented as
$$\varphi_{j}(t)=\me~\varphi_{j}(t)+\int_{0}^{t}\phi_{j}(s,t)dW(s),\quad a.s., \ a.e. \ t\in[0,T].$$

Consequently,  we have
\begin{eqnarray*}
& &\me~\int_{0}^{T}\int_{0}^{T}\Big|\phi_{j}(s,t)-\phi_{m}(s,t)\Big|^2dsdt
=
\me~\int_{0}^{T}\int_{0}^{t}\Big|\phi_{j}(s,t)-\phi_{m}(s,t)\Big|^2dsdt\\
&=&
\int_{0}^{T}\me~\Big|\int_{0}^{t}\Big[\phi_{j}(s,t)-\phi_{m}(s,t)\Big]dW(s)\Big|^2dt\\
&=&\int_{0}^{T}
\me~\Big|\varphi_{j}(t)-\varphi_{m}(t)- \me~\big[\varphi_{j}(t)-\varphi_{m}(t)\big]\Big|^2dt
\le
4\me~\int_{0}^{T}\Big|\varphi_{j}(t)-\varphi_{m}(t)\Big|^2dt.
\end{eqnarray*}
Since $\varphi_{j}(\cdot)$ converges strongly to $\varphi(\cdot)$, $\{\phi_{j}(\cdot,\cdot)\}_{j=1}^{\infty}$ is a Cauchy sequence in $L^{2}(0,T;$  $L_{\mmf}^{2}([0,T]\times\Omega; \mrn))$. Hence, there exists a $\phi(\cdot,\cdot)\in L^{2}(0,T;  L_{\mmf}^{2}([0,T]\times\Omega; \mrn))$
such that
$$\me~\int_{0}^{T}\int_{0}^{T}\Big|\phi_{j}(s,t)-\phi(s,t)\Big|^2dsdt
\to 0,\qquad\hbox{as } j\to \infty,$$
and
\begin{eqnarray*}
& &\me~\int_{0}^{T}\Big|\varphi(t)-\me~\varphi(t)-\int_{0}^{t}\phi(s,t)dW(s) \Big|^2dt\\
&= &
\me~\int_{0}^{T}\Big|\varphi(t)-\varphi_{j}(t)+\varphi_{j}(t)
-\me~\varphi(t)+\me~\varphi_{j}(t)-\me~\varphi_{j}(t)\\
& &-\int_{0}^{t}\phi(s,t)dW(s) +\int_{0}^{t}\phi_{j}(s,t)dW(s)
-\int_{0}^{t}\phi_{j}(s,t)dW(s)\Big|^2dt\\
&\le&
C\me~\int_{0}^{T}\Big|\varphi(t)-\varphi_{j}(t)\Big|^2dt
+C\int_{0}^{T}\Big|\me~\varphi(t)-\me~\varphi_{j}(t)\Big|^2dt\\
& &
+C\me~\int_{0}^{T}\Big|\int_{0}^{t}\phi(s,t)dW(s) -\int_{0}^{t}\phi_{j}(s,t)dW(s)
\Big|^2dt\\
& &
+C\me~\int_{0}^{T}\Big|\varphi_{j}(t)-\me~\varphi_{j}(t)-\int_{0}^{t}\phi_{j}(s,t)dW(s)    \Big|^2dt\\
&\le&
C\me~\int_{0}^{T}\Big|\varphi(t)-\varphi_{j}(t)\Big|^2dt
+C\me~\int_{0}^{T}\int_{0}^{T}\Big|\phi(s,t) -\phi_{j}(s,t)\Big|^2dsdt\\
& & \to 0,\qquad \hbox{as }j\to \infty.
\end{eqnarray*}
Therefore,  (\ref{martingale exp formulation}) holds.
\end{proof}

Also, we need the following simple result.

\begin{lemma}\label{bounded of S convex}
Let (C1)--(C2) hold. Then $\ms(\cdot)
\in L_{\mmf}^{4}(\Omega; L^{2}(0,T; \mr^{m\times n}))$.
\end{lemma}
\begin{proof}
We only need to prove that
$$ \me~\Big[\int_{0}^{T}\big|\ms(t)\big|^2dt\Big]^{2}<\infty.$$
By {\em (C1)--(C2)},
$$|f_{xu}(t)|\le C,\quad a.s.,\ a.e. \ t\in [0,T],$$
and, for $\varphi=b,\ \sigma$,
$$|\varphi_{x}(t)|+|\varphi_{u}(t)|+|\varphi_{xu}(t)|\le C,\quad a.s.,\ a.e. \ t\in [0,T].$$
Therefore,
\begin{eqnarray*}
& &\me~\Big[\int_{0}^{T}\big|\ms(t)\big|^2dt\Big]^{2}\\
&=& \me~\Big[\int_{0}^{T}\big| H_{xu}(t,\bar{x}(t),\bar{u}(t),P_{1}(t),Q_{1}(t))
+b_{u}(t,\bar{x}(t),\bar{u}(t))^{\top}P_{2}(t)\\
& &\quad\
+\sigma_{u}(t,\bar{x}(t),\bar{u}(t))^{\top}Q_{2}(t)
+\sigma_{u}(t,\bar{x}(t),\bar{u}(t))^{\top}
P_{2}(t)\sigma_{x}(t,\bar{x}(t),\bar{u}(t))\big|^2dt\Big]^{2}\\
&\le& C + C\me~\Big[\int_{0}^{T}\big(|P_{1}(t)|^2+|Q_{1}(t)|^2+|P_{2}(t)|^2
+|Q_{2}(t)|^2\big)dt\Big]^{2}\\
&\le& C + C\big(\|P_{1}\|_{\infty,4}^{4}+\|Q_{1}\|_{2,4}^{4}
+\|P_{2}\|_{\infty,4}^{4}+\|Q_{2}\|_{2,4}^{4}\big)\\
&<& \infty,
\end{eqnarray*}
which completes the proof of Lemma \ref{bounded of S convex}.
\end{proof}

By Lemma \ref{bounded of S convex}, $\ms(\cdot)
\in L_{\mmf}^{2}(\Omega; L^{2}(0,T; \mr^{m\times n}))$. Then, by our assumption (C1) and Lemma \ref{martingale exp for L2 function}, for any $v\in U$, there exists a $\phi_{v}(\cdot,\cdot)\in  L^{2}(0,T; L_{\mmf}^{2}([0,T]\times\Omega; \mrn))$ such that for a.e. $t\in[0,T]$,
\begin{equation}\label{martingale exp for convex S(v-u)}
\ms(t)^{\top}(v-\bar{u}(t))
=\me ~\Big[\ms(t)^{\top}(v-\bar{u}(t))\Big]
 +\int_{0}^{t} \phi_{v}(s,t)dW(s),\quad a.s.
\end{equation}

Using (\ref{martingale exp for convex S(v-u)}), we obtain the following second-order necessary condition, which is pointwise with respect to the time variable (but it is still in the integral form with respect to the sample point $\omega$).

\begin{theorem}\label{2orderconditionthconvex pointwise with t}
Let (C1)--(C2) hold. If $\bar{u}(\cdot)$ is a singular optimal control in the classical sense, then for any $v\in U$, it holds that
\begin{eqnarray}\label{2orderconditionconvex martingleexp}
& &\me~\inner{\ms(\t)b_{u}(\t)(v-\bar{u}(\t))}{v-\bar{u}(\t)}\nonumber\\
& &\quad + \partial^{+}_{\t}\big(\ms(\t)^{\top}(v-\bar{u}(\t)),
\sigma_{u}(\t)(v-\bar{u}(\t))\big) \le 0, \  a.e.\ \t\in [0,T],
\end{eqnarray}
where,
\begin{equation}\label{limsup martingleexp for convex}
\quad\quad\begin{array}{ll}
\displaystyle\partial^{+}_{\t}\big(\ms(\t)^{\top}(v-\bar{u}(\t)),
\sigma_{u}(\t)(v-\bar{u}(\t))\big)\\[3mm]\displaystyle
:=2\limsup_{\theta\to 0^+}\frac{1}{\theta^2}\me\int_{\t}^{\t+\theta}
\int_{\t}^{t}\Big\langle\phi_{v}(s,t),
\Phi(\t)\Phi(s)^{-1}\sigma_{u}(s)(v-\bar{u}(s))\Big\rangle dsdt,\quad\quad
\end{array}
\end{equation}
$\phi_{v}(\cdot,\cdot)$ is determined by (\ref{martingale exp for convex S(v-u)}), and
$\Phi(\cdot)$ is the solution to the stochastic differential equation (\ref{Phi}).
\end{theorem}

The proof of Theorem \ref{2orderconditionthconvex pointwise with t} will be given in Subsection 4.1.

\subsection{Second-order necessary condition in term of Malliavin derivative}

In Theorem \ref{2orderconditionthconvex pointwise with t} we obtain a second-order necessary condition in term of martingale representation. From the martingale representation theorem, we only know that, for any $v\in U$, $\phi_{v}(\cdot,\cdot)\in L^{2}(0,T; L_{\mmf}^{2}([0,T]\times\Omega; \mrn))$, and hence, for each $\tau\in [0,T]$,  the function
$$\varphi(s,t):=\me ~\inner{\phi_{v}(s,t)}{\Phi(\t)\Phi(s)^{-1}\sigma_{u}(s)(v-\bar{u}(s))},\quad (s,t)\in [0,T]\times[0,T]$$
is in $L^1([0,T]\times[0,T])$. However, the condition  $\varphi(\cdot,\cdot)\in L^{1}([0,T]\times[0,T])$ is not sufficient to ensure that, for a.e. $\t\in[0,T]$, the  limit
\begin{equation}\label{limitcontexample}
\lim_{\theta\to 0^+}\frac{1}{\theta^2}\int_{\t}^{\t+\theta}
\int_{\t}^{t}\varphi(s,t)dsdt
\end{equation}
exists.

\begin{example}\label{example inflim nonequ suplim}
Let $a_n=\frac{2}{3^n}$, $n=0, 1, 2, \cdots$. Then, $\sum_{n=1}^{\infty}a_{n}=1$ and $\sum_{k=n+1}^{\infty}a_{k}=\frac{a_n}{2}$. Let $T=\sqrt{2}$ and define $\varphi(\cdot,\cdot)\in L^{1}([0,\sqrt{2}]\times[0,\sqrt{2}])$ as follows:
$$\varphi(s,t)=\left\{
\begin{array}{l}
1, \ \ \ \ (s,t)\in  \big([0,\sqrt{2}]\times[0,\sqrt{2}]\big)\bigcap
\big\{\frac{a_{n}}{2}\le\frac{t-s}{\sqrt{2}}< a_n ,\ n=1,\ 2,\ \cdots\big\},\\
-1, \ \ (s,t)\in \big([0,\sqrt{2}]\times[0,\sqrt{2}]\big)\bigcap\big\{ a_n\le\frac{t-s}{\sqrt{2}}
< \frac{a_{n-1}}{2},\ n=1,\ 2,\ \cdots\big\},\\
0,  \ \ \ \ \mbox{otherwise}.
\end{array}\right.$$

Fixed a $\t \in[0,\sqrt{2})$ arbitrarily. If $\theta_{n}=\frac{\sqrt{2}a_{n-1}}{2}$, $\t+\theta_{n}\le \sqrt{2}$, then
$$\lim_{n\to \infty}\frac{1}{\theta_{n}^{2}}\int_{\t}^{\t+\theta_{n}}
\int_{\t}^{t}\varphi(s,t)dsdt
=\lim_{n\to \infty} \frac{\sum_{k=n}^{\infty}(\frac{\sqrt{2}a_{k}}{2})^{2}}
{(\frac{\sqrt{2}a_{n-1}}{2})^{2}}
=\lim_{n\to \infty} \frac{\frac{1}{4\cdot 9^{n-1}}}
{\frac{2}{9^{n-1}}}
=\frac{1}{8}.$$
On the other hand, if $\theta_{n}=\sqrt{2}a_{n}$, $\t+\theta_{n}\le \sqrt{2}$, then
\begin{eqnarray*}
\lim_{n\to \infty}\frac{1}{\theta_{n}^{2}}\int_{\t}^{\t+\theta_{n}}
\int_{\t}^{t}\varphi(s,t)dsdt
&=&\lim_{n\to \infty} \frac{\frac{1}{2}(\frac{\sqrt{2}a_{n}}{2})^{2}
+\sum_{k=n}^{\infty}(\frac{\sqrt{2}a_{k+1}}{2})^{2}}
{(\sqrt{2}a_{n})^{2}}\\
&=&\lim_{n\to \infty} \frac{\frac{5}{4\cdot 9^{n}}}
{\frac{8}{9^{n}}}
=\frac{5}{32}.
\end{eqnarray*}
\end{example}

\begin{example}
Let $T=1$. Define
$$\varphi(s,t)
=\left\{
\begin{array}{l}
       0, \qquad\qquad\quad\ \ \  t\le s,\ s,t\in [0,1],\\
       -\frac{1}{(t-s)^{\frac{1}{2}}}, \qquad \quad t> s,\ s,t\in [0,1].\\
\end{array}\right.$$
Obviously, $\varphi\in L^{1}([0,1]\times[0,1])$.
But, for any $\t\in[0,1)$ and $\theta>0$ satisfying $\t+\theta\le 1$,
$$\frac{\int^{\t+\theta}_{\t}\int^{t}_{\t}\varphi(s,t)dsdt}{\theta^2}
=\frac{-\int^{\t+\theta}_{\t}2(t-\t)^{\frac{1}{2}}dt}{\theta^2}
=\frac{-\frac{4}{3}\theta^{\frac{3}{2}}}{\theta^2}
\to -\infty,\quad (\theta \to 0^+).
$$
\end{example}

The above two examples show that, in general, the superior limit
\begin{eqnarray*}
\limsup_{\theta\to0^+}\frac{1}{\theta^2}\me\int_{\t}^{\t+\theta}
\int_{\t}^{t}\Big\langle\phi_{v}(s,t),
\Phi(\t)\Phi(s)^{-1}\sigma_{u}(s)(v-\bar{u}(s))\Big\rangle dsdt
\end{eqnarray*}
(in (\ref{limsup martingleexp for convex})) cannot be refined to be the limit,
and even worse, this superior limit
may be equal to $-\infty$. If the superior limit in (\ref{limsup martingleexp for convex}) is equal to $-\infty$ for a.e. $\t\in [0,T]$, the second-order necessary condition (\ref{2orderconditionconvex martingleexp}) turns out to be trivial.
On the other hand, even this superior limit is finite for a.e. $\t\in [0,T]$,
it is still difficult to obtain the continuity of the function
$$v\mapsto\partial^{+}_{\t}\big(\ms(\t)^{\top}(v-\bar{u}(\t)),
\sigma_{u}(\t)(v-\bar{u}(\t))\big).$$

All the problems  mentioned in the above are caused by the lack of further information for $\phi_{v}(\cdot,\cdot)$. If both $\ms(\cdot)$ and $\bar{u}(\cdot)$ are regular enough, the function $\phi_{v}(\cdot,\cdot)$ has an explicit representation and then we can improve the result obtained in Theorem
\ref{2orderconditionthconvex pointwise with t}. To this end, we assume that

\begin{enumerate}
\item [{\bf (C3)}]
$$\bar{u}(\cdot)\in\ml_{2,\mmf}^{1,2}(\mrm),\ \ms(\cdot)\in \ml_{2,\mmf}^{1,2}(\mr^{m\times n})\cap L^{\infty}([0,T]\times\Omega;\mr^{m\times n}).$$
\end{enumerate}

We have the following pointwise second-order necessary condition for singular optimal controls.

\begin{theorem}\label{2orderconditionth}
Let (C1)--(C3) hold. If $\bar{u}(\cdot)$ is a singular optimal control in the classical sense, then for a.e. $\t\in [0,T]$, it holds that
\begin{eqnarray}\label{2orderconditionconvex}
& &\inner{\ms(\t)b_{u}(\t)(v-\bar{u}(\t))}{v-\bar{u}(\t)}\nonumber\\
& &\ \ + \inner{\nabla \ms(\t)\sigma_{u}(\t)(v-\bar{u}(\t))}
{v-\bar{u}(\t)}\\
& &\ \ - \inner{\ms(\t)\sigma_{u}(\t)(v-\bar{u}(\t))}
{\nabla\bar{u}(\t)}
\le 0, \quad \forall \ v\in U, \ a.s.\nonumber
\end{eqnarray}
\end{theorem}

The proof of Theorem \ref{2orderconditionth} will be given in Subsection 4.2.

\begin{remark}
In some special cases, the regularity assumption on $\ms(\cdot)$ holds automatically. One of them is the linear quadratic optimal control problem with convex control constraints.
In this case, the functions $b$, $\sigma$, $f$ and $h$ in (\ref{controlsys})-(\ref{costfunction}) are given as follows:
\begin{eqnarray*}
& &b(t,x,u)=A(t)x+B(t)u,\quad \sigma(t,x,u)=C(t)x+D(t)u,\quad h(x)=\frac{1}{2}\inner{Gx}{x},\\
& &f(t,x,u)=\frac{1}{2}\big[\inner{R(t)x}{x}+2\inner{M(t)x}{u}+\inner{N(t)u}{u}\big],\quad (t,x,u)\in [0,T]\times\mrn\times\mrm,
\end{eqnarray*}
where $A(\cdot), C(\cdot)\in C([0,T]; \mr^{n\times n})$, $B(\cdot), D(\cdot)\in C([0,T]; \mr^{n\times m})$, $R(\cdot)\in C([0,T];$ $\mathbf{S}^n)$, $M(\cdot)\in C([0,T]; \mr^{m\times n})$ and $N(\cdot)\in C([0,T]; \mathbf{S}^m)$ are deterministic matrix-valued functions, and $G\in \mathbf{S}^n$ is a (deterministic) matrix.

Indeed, for this problem, the second-order adjoint equation is
\qquad\begin{equation}\label{second order adjoint for LQ}
\left\{
\begin{array}{l}
dP_{2}(t)=-\Big[A(t)^{\top}P_{2}(t)+P_{2}(t)A(t)
+C(t)^{\top}P_{2}(t)C(t)+C(t)^{\top}Q_{2}(t)\\
\qquad\qquad\qquad
+ Q_{2}(t)C(t)-R(t)\Big]dt+Q_{2}(t)dW(t),\ \ t\in[0,T],\\
P_{2}(T)=-G.
\end{array}\right.
\end{equation}
Since $A(\cdot)$, $B(\cdot)$, $C(\cdot)$,  $D(\cdot)$, $R(\cdot)$, $M(\cdot)$, $N(\cdot)$ and $G$ are deterministic, the equatio(\ref{second order adjoint for LQ}) admits a unique deterministic solution $(P_{2}(\cdot),0)$,  where $P_{2}(\cdot)$ is the solution to the following differential equation
\begin{equation}\label{second order adjoint for LQ ordinary}
\qquad\left\{
\begin{array}{l}
\dot{P}_{2}(t)=-A(t)^{\top}P_{2}(t)-P_{2}(t)A(t)-C(t)^{\top}P_{2}(t)C(t)
+R(t),\ t\in[0,T],\\
P_{2}(T)=-G.
\end{array}\right.
\end{equation}
Hence, for this case,
$$\ms(\cdot)=-M(\cdot)+B(\cdot)^{\top}P_{2}(\cdot)+D(\cdot)^{\top}P_{2}(\cdot)C(\cdot)$$
is a deterministic continuous matrix-valued function, hence it belongs to the space $\ml^{1,2}_{2,\mmf}(\mr^{m\times n})\cap L^{\infty}([0,T]\times\Omega;\mr^{m\times n})$.

In general, to obtain the regularity of $\ms(\cdot)$, we need the regularity of $(\bar{u}(\cdot),\bar{x}(\cdot))$, $(P_{1}(\cdot),Q_{1}(\cdot))$ and $(P_{2}(\cdot),Q_{2}(\cdot))$. From the regularity results for solutions to stochastic differential equations (see \cite{Peng97} and \cite{Nualart06}), the optimal control $\bar{u}(\cdot)$ needs to be regular enough.
In the deterministic setting, the regularity of optimal controls has been studied by many authors (see \cite{Cannarsa09,Clarke90} and references cited therein). However, to the best of our knowledge, there exists no reference addressing the regularity of stochastic optimal controls. We will discuss this topic in our forthcoming paper.
\end{remark}

To end this section, we give two examples to explain how to distinguish singular optimal controls from others by using the pointwise second-order necessary conditions established in Theorem \ref{2orderconditionth}.

\begin{example}
Let $n=m=1$, $T=1$, $U=[-1, 1]$. Consider the following one-dimensional control system
\begin{equation}\nonumber
\left\{
\begin{array}{l}
dx(t)=u(t)dt+u(t)dW(t),\qquad t\in[0,1],\\
x(0)=0
\end{array}\right.
\end{equation}
and the cost functional
$$J(u(\cdot))=\frac{1}{2}\me\int_{0}^{1} |u(t)|^2dt-\frac{1}{2}\me~|x(1)|^2.$$
For this optimal control problem, the Hamiltonian is given by
 $$H(t,x,u,y_1,z_1)=y_1u+z_1u-\frac{1}{2}u^2,\quad (t,x,u,y_1,z_1)\in [0,1]\times\mr\times U\times\mr\times\mr.$$
Let $(\bar{x}(t),\bar{u}(t))\equiv(0,0)$. The corresponding two adjoint equations are
\begin{equation}\nonumber
\left\{
\begin{array}{l}
dP_{1}(t)=Q_{1}(t)dW(t), \ t\in[0,1],\qquad\qquad \qquad\qquad \qquad\\
P_{1}(1)=0,
\end{array}\right.
\end{equation}
and
\begin{equation}\nonumber
\left\{
\begin{array}{l}
dP_{2}(t)=Q_{2}(t)dW(t), \ t\in[0,1],\qquad\qquad \qquad\qquad \qquad\\
P_{2}(1)=1.
\end{array}\right.
\end{equation}
Obviously,
$$(P_{1}(t),Q_{1}(t))\equiv(0,0),\ \qquad (P_{2}(t),Q_{2}(t))\equiv(1,0).$$
Then, we have for all $(t,\omega)\in [0,1]\times\Omega$,
$$H_{u}(t,\bar{x}(t),\bar{u}(t),P_{1}(t),Q_{1}(t))=0,$$
and
$$H_{uu}(t,\bar{x}(t),\bar{u}(t),P_{1}(t),Q_{1}(t))
+\sigma_{u}(t,\bar{x}(t),\bar{u}(t))^{\top}
P_{2}(t)\sigma_{u}(t,\bar{x}(t),\bar{u}(t))=0.$$
That is, $\bar{u}(t)\equiv 0$ is a singular control in the classical sense. Let $\hat{u}(t)\equiv-1$, we have
 $$-\frac{1}{2}=J(\hat{u}(\cdot))<J(\bar{u}(\cdot))=0.$$
Therefore, $\bar{u}(t)\equiv 0$ is not an optimal control.

Now, we show that $\bar{u}(t)\equiv 0$ does not satisfy the second-order necessary condition (\ref{2orderconditionconvex}).
Actually,
$$\ms(t)\equiv1,\quad \nabla \ms(t)\equiv0,\quad \nabla \bar{u}(t)\equiv0.$$
Let $v= 1$, we find that
\begin{eqnarray*}
& &\inner{\ms(\t)b_{u}(\t)(v-\bar{u}(\t))}{v-\bar{u}(\t)}\\
& &\ \ +\inner{\nabla \ms(\t)\sigma_{u}(\t)(v-\bar{u}(\t))}
{v-\bar{u}(\t)}\\
& &\ \ - \inner{\ms(\t)\sigma_{u}(\t)(v-\bar{u}(\t))}
{\nabla\bar{u}(\t)}\\
&=&1>0, \quad \forall\; (\t,\omega)\in [0,1]\times\Omega.
\end{eqnarray*}
Hence, the condition (\ref{2orderconditionconvex}) fails at $v=1$.
\end{example}

\begin{example}
Let $n=m=1$, $U=[-1,1]\times[-1,1]$. Consider the control system
\begin{equation}\nonumber
\left\{
\begin{array}{l}
dx(t)= Bu(t)dt+Du(t)dW(t),
\qquad \ \ \ t\in[0,T], \qquad\qquad \qquad\\
x(0)=0
\end{array}\right.
\end{equation}
with the following cost functional
$$J(u(\cdot))=\frac{1}{2}\me\inner{Gx(T)}{x(T)},$$
where
$$B=\left[ \begin{array}{cc}
          1 & 0 \\
          0 & 0 \\
\end{array} \right],\qquad
        D=\left[ \begin{array}{cc}
          0 & 0 \\
          0 & 1 \\
        \end{array} \right],\qquad
        G=\left[ \begin{array}{cc}
          1 & 0 \\
          0 & 0 \\
        \end{array} \right].$$

For this optimal control problem, the Hamiltonian is given by
$$H(t,x,u,y_{1}, z_{1})\!=\!\inner{y_1}{Bu}\!+\!\inner{z_1}{Du},\quad (t,x,u,y_1,z_1)\!\in\! [0,T]\!\times\!\mr^{2}\!\times\! U\!\times\!\mr^{2}\!\times\!\mr^{2}.$$

Clearly, $(\bar{x}(t),\bar{u}(t))\equiv(0,0)$ is an optimal pair, and
the corresponding adjoint equations are respecitvely
\begin{equation}\nonumber
\left\{
\begin{array}{l}
dP_{1}(t)= Q_{1}(t)dW(t),
\qquad \ \ \ t\in[0,T], \qquad\qquad \qquad\\
P_{1}(T)=0
\end{array}\right.
\end{equation}
and
\begin{equation}\nonumber
\left\{
\begin{array}{l}
dP_{2}(t)= Q_{2}(t)dW(t),
\qquad \ \ \ t\in[0,T], \qquad\qquad \qquad\\
P_{2}(T)=-G.
\end{array}\right.
\end{equation}

Obviously,
$(P_{1}(t),Q_{1}(t))\equiv(0,0)$, $(P_{2}(t),Q_{2}(t))\equiv(-G,0)$,
and
$$H_{u}(t,\bar{x}(t),\bar{u}(t),P_{1}(t), Q_{1}(t))\equiv0,$$
$$H_{uu}(t,\bar{x}(t),\bar{u}(t),P_{1}(t), Q_{1}(t))+D^{\top}P_{2}(t)D\equiv0.$$
Therefore, $\bar{x}(t)\equiv0$ is a singular optimal control in the classical sense.

Since for this case,
$$\ms(t)\equiv-B^{\top}G,\quad \nabla \ms(t)\equiv 0,\quad \nabla \bar{u}(t)\equiv 0,$$
we have
\begin{eqnarray*}
& &\inner{\ms(\t)b_{u}(\t)(v-\bar{u}(\t))}{v-\bar{u}(\t)}\\
& &+\inner{\nabla \ms(\t)\sigma_{u}(\t)(v-\bar{u}(\t))}
{v-\bar{u}(\t)}\\
& & - \inner{\ms(\t)\sigma_{u}(\t)(v-\bar{u}(\t))}
{\nabla\bar{u}(\t)}\\
&=&-\inner{B^{\top}GBv}{v}\le 0, \quad \forall \; v\in U,\quad \forall\; (t,\omega)\in [0,T]\times\Omega.
\end{eqnarray*}
That is, the necessary condition (\ref{2orderconditionconvex}) holds.
\end{example}

\section{Proofs of the main results}

This section is devoted to proving Theorems \ref{2orderconditionthconvex pointwise with t} and \ref{2orderconditionth}. Firstly, we show a technical result.

\begin{lemma}\label{technical lemma}
Let $\Phi(\cdot),\ \Psi(\cdot)\in L_{\mmf}^{2}(\Omega;L^{2}(0,T;\mrn))$. Then, for a.e. $\t\in [0,T)$,
\begin{equation}\label{technical lemma limit1}
\lim_{\theta\to 0^+}\frac{1}{\theta^2}\me\int_{\t}^{\t+\theta}\Big\langle\Phi(\t),
\int_{\t}^{t} \Psi(s)ds\Big\rangle dt
=\frac{1}{2}\me\inner{\Phi(\t)}{\Psi(\t)},
\end{equation}
\begin{equation}\label{technical lemma limit2}
\lim_{\theta\to 0^+}\frac{1}{\theta^2}\me\int_{\t}^{\t+\theta}
\Big\langle\Phi(t), \int_{\t}^{t} \Psi(s)ds \Big\rangle dt
=\frac{1}{2}\me\inner{\Phi(\t)}{\Psi(\t)}.
\end{equation}
\end{lemma}
\begin{proof}
The equality (\ref{technical lemma limit1}) is a corollary of the Lebesgue differentiation theorem. Now, we prove (\ref{technical lemma limit2}). For any $\t\in [0,T)$, let $\theta>0$ and $\t+\theta<T$. By the Lebesgue differentiation theorem, we have
$$
\lim_{\theta\to 0^+}\frac{1}{\theta}\int_{\t}^{\t+\theta}
\me~\big|\Phi(t)-\Phi(\t)\big|^2dt
= 0,\ \ \  a.e.\ \t\in [0,T),
$$
and
$$
\lim_{\theta\to 0^+}\frac{1}{\theta^{2}}\me\int_{\t}^{\t+\theta}
\int_{\t}^{t}\big|\Psi(s)\big|^2dsdt
= \frac{1}{2}\me\big|\Psi(\t)\big|^2,\ \ \  a.e.\ \t\in [0,T).
$$
Therefore,
\begin{eqnarray}\label{technical lemma equ4}
& &\lim_{\theta\to 0^+}\Big|\frac{1}{\theta^2}\me\int_{\t}^{\t+\theta}
\Big\langle \Phi(t)-\Phi(\t), \int_{\t}^{t} \Psi(s)ds\Big\rangle dt\Big|\\
&\le& \lim_{\theta\to 0^+}\frac{1}{\theta^2}
\Big[\int_{\t}^{\t+\theta}
\me~\big|\Phi(t)-\Phi(\t)\big|^2dt\Big]^{\frac{1}{2}}
\Big[\int_{\t}^{\t+\theta}
(t-\t)\me~\int_{\t}^{t}\big|\Psi(s) \big|^2ds dt\Big]^{\frac{1}{2}}
\nonumber\\
&\le&\lim_{\theta\to 0^+}\frac{1}{\theta^{\frac{3}{2}}}\Big[\int_{\t}^{\t+\theta}
\me~\big|\Phi(t)-\Phi(\t)\big|^2dt\Big]^{\frac{1}{2}}
\Big[\int_{\t}^{\t+\theta}
\me~\int_{\t}^{t}\big|\Psi(s) \big|^2ds dt\Big]^{\frac{1}{2}}
\nonumber\\
& =& 0,\ \qquad a.e.\ \t\in [0,T).\nonumber
\end{eqnarray}
Combining (\ref{technical lemma equ4})  and (\ref{technical lemma limit1}), we obtain (\ref{technical lemma limit2}).
This completes the proof of Lemma \ref{technical lemma}.
\end{proof}

\subsection{Proof of Theorem \ref{2orderconditionthconvex pointwise with t}}
%
For any $v\in U$, $\t\in[0,T)$ and $\theta\in (0,T-\t)$, let $E_{\theta}=[\t, \t+\theta)$ and $u(\cdot)$ be defined by (\ref{needle variation}).
Then, $v(\cdot)=u(\cdot)-\bar{u}(\cdot)=(v-\bar{u}(\cdot))\chi_{E_{\theta}}(\cdot)$
and the corresponding solution $y_{1}(\cdot)$ to the equation (\ref{firstvariequconvex}) is given by
\begin{eqnarray}\label{y1(t) in2condtion convex mart}
\qquad y_{1}(t)&=&\Phi(t)\int_{0}^{t}\Phi(s)^{-1}
\big(b_{u}(s)-\sigma_{x}(s)\sigma_{u}(s)
\big)\big(v-\bar{u}(s)\big)\chi_{E_{\theta}}(s)ds\\
& & +\Phi(t)\int_{0}^{t}\Phi(s)^{-1}\sigma_{u}(s)(v-\bar{u}(s))
\chi_{E_{\theta}}(s)dW(s).\nonumber
\end{eqnarray}

Substituting $v(\cdot)=(v-\bar{u}(\cdot))\chi_{E_{\theta}}(\cdot)$ and (\ref{y1(t) in2condtion convex mart})
into (\ref{integraltype 2ordercondition}), we have
\begin{eqnarray}\label{limit of S(t)y1(t)}
\qquad 0&\ge&\frac{1}{\theta^2}\me\int_{\t}^{\t+\theta}
\inner{\ms(t)y_{1}(t)}{v-\bar{u}(t)}dt\nonumber\\
&=&\frac{1}{\theta^2}\me\int_{\t}^{\t+\theta}
\Big\langle \ms(t)\Phi(t)\int_{\t}^{t}\Phi(s)^{-1}
\big(b_{u}(s)-\sigma_{x}(s)\sigma_{u}(s)
\big)\cdot\nonumber\\
& &\qquad\qquad\qquad\qquad\qquad\qquad\qquad\qquad\
(v-\bar{u}(s))ds,
v-\bar{u}(t)\Big\rangle dt\\
& &+\frac{1}{\theta^2} \me\int_{\t}^{\t+\theta}
\Big\langle\ms(t)\Phi(t)\int_{\t}^{t}
\Phi(s)^{-1}\sigma_{u}(s)\cdot\nonumber\\
& &\qquad\qquad\qquad\qquad\qquad\qquad\qquad
(v-\bar{u}(s))dW(s),
v-\bar{u}(t)\Big\rangle dt.\nonumber
\end{eqnarray}

By Lemma \ref{technical lemma}, we have, for a.e. $\t\in [0,T)$,
\begin{eqnarray}\label{limit of S(t)y1(t)part1}
& &\lim_{\theta\to 0^+}\frac{1}{\theta^2}\me\int_{\t}^{\t+\theta}
\Big\langle \ms(t)\Phi(t)\int_{\t}^{t}\Phi(s)^{-1}
\big(b_{u}(s)-\sigma_{x}(s)\sigma_{u}(s)\big)\cdot\nonumber\\
& &\qquad\qquad\qquad\qquad\qquad\qquad\qquad\qquad\qquad
(v-\bar{u}(s))ds,
v-\bar{u}(t)\Big\rangle dt\nonumber\\
&=&\frac{1}{2}\me \inner{\ms(\t)\big(b_{u}(\t)-\sigma_{x}(\t)\sigma_{u}(\t)
\big)(v-\bar{u}(\t))}
{v-\bar{u}(\t)}.
\end{eqnarray}

On the other hand, by (\ref{Phi}), we have
\begin{eqnarray}\label{limit of S(t)y1(t)part2}
& &\frac{1}{\theta^2}\int_{\t}^{\t+\theta} \me~
\Big\langle \ms (t)\Phi(t)\int_{\t}^{t}
\Phi(s)^{-1}\sigma_{u}(s)(v-\bar{u}(s))dW(s),
v-\bar{u}(t) \Big\rangle dt\\
&=&\frac{1}{\theta^2}\int_{\t}^{\t+\theta}\me~
\Big\langle\ms (t)\Phi(\t)\int_{\t}^{t}
\Phi(s)^{-1}\sigma_{u}(s)(v-\bar{u}(s))dW(s),
v-\bar{u}(t) \Big\rangle dt\nonumber\\
& &+\frac{1}{\theta^2}\int_{\t}^{\t+\theta}\me~
\Big\langle\ms (t)\int_{\t}^{t}b_{x}(s)\Phi(s)ds\cdot\nonumber\\
& &\qquad\qquad\qquad\qquad\quad
\int_{\t}^{t}\Phi(s)^{-1}\sigma_{u}(s)(v-\bar{u}(s))dW(s),
v-\bar{u}(t)\Big\rangle dt\nonumber\\
& &+\frac{1}{\theta^2}\int_{\t}^{\t+\theta}\me~
\Big\langle\ms(t)\int_{\t}^{t}
\sigma_{x}(s)\Phi(s)dW(s)\cdot\nonumber\\
& &\qquad\qquad\qquad\qquad\
\int_{\t}^{t}\Phi(s)^{-1}\sigma_{u}(s)(v-\bar{u}(s)) dW(s),
v -\bar{u}(t)\Big\rangle dt.\nonumber
\end{eqnarray}

Substituting (\ref{martingale exp for convex S(v-u)}) into the fist term of the right hand of (\ref{limit of S(t)y1(t)part2}),  we get that
\begin{eqnarray}\label{th31 equ1 martingle}
\quad& &\limsup_{\theta\to 0^+}\frac{1}{\theta^2}\int_{\t}^{\t+\theta}\me~
\Big\langle\ms (t)\Phi(\t)\int_{\t}^{t}
\Phi(s)^{-1}\sigma_{u}(s)(v-\bar{u}(s))dW(s),
v-\bar{u}(t)\Big\rangle dt\\
&=&\limsup_{\theta\to 0^+}\frac{1}{\theta^2}\int_{\t}^{\t+\theta}
\me~\Big\langle\int_{\t}^{t}\Phi(\t)\Phi(s)^{-1}\sigma_{u}(s)
(v-\bar{u}(s))dW(s),\nonumber\\
& & \qquad\qquad\qquad\qquad\qquad\qquad\qquad\qquad\qquad\qquad\qquad
\me~\big[\ms(t)^{\top}(v-\bar{u}(t))\big]
\Big\rangle dt\nonumber\\
& & \qquad+\limsup_{\theta\to 0^+}\frac{1}{\theta^2}\int_{\t}^{\t+\theta}
\me~\Big\langle \int_{\t}^{t}\Phi(\t)\Phi(s)^{-1}\sigma_{u}(s)
(v-\bar{u}(s))dW(s),\nonumber\\
& & \qquad\qquad\qquad\qquad\quad\qquad\qquad\qquad\qquad\qquad\qquad\qquad
\int_{0}^{t} \phi_{v}(s,t) dW(s) \Big\rangle dt\nonumber\\
&=&\limsup_{\theta\to 0^+}\frac{1}{\theta^2}\int_{\t}^{\t+\theta}\int_{\t}^{t}
\me~\Big\langle\Phi(\t)\Phi(s)^{-1}\sigma_{u}(s)
(v-\bar{u}(s)),\phi_{v}(s,t) \Big\rangle dsdt\nonumber\\
&=&\frac{1}{2}\partial^{+}_{\t}\big(\ms(\t)^{\top}(v-\bar{u}(\t)),
\sigma_{u}(\t)(v-\bar{u}(\t))\big), \qquad \forall\ \t\in [0,T).\nonumber
\end{eqnarray}

Next, by Lemma \ref{bounded of S convex}, $\ms(\cdot)\in L^{4}(\Omega;L^{2}(0,T;\mr^{m\times n}))\subset L^{2}(\Omega;L^{2}(0,T;\mr^{m\times n}))$. Then, by Condition (C1), we have
\begin{eqnarray*}
& &\lim_{\theta\to 0^+}\Big|\frac{1}{\theta^2}\int_{\t}^{\t+\theta}\me~
\Big\langle\ms(t)\int_{\t}^{t}b_{x}(s)\Phi(s)ds\cdot\\
& &\qquad\qquad\qquad\qquad\qquad\qquad
 \int_{\t}^{t}\Phi(s)^{-1}\sigma_{u}(s)
(v-\bar{u}(s))dW(s),
v-\bar{u}(t)\Big\rangle dt\Big|\\
&\le&\lim_{\theta\to 0^+}\frac{C}{\theta^2}\int_{\t}^{\t+\theta}\me~
\Big|\ms(t)\int_{\t}^{t}b_{x}(s)\Phi(s)ds
\int_{\t}^{t}\Phi(s)^{-1}\sigma_{u}(s)
(v-\bar{u}(s))dW(s)\Big|dt\\
&\le&\lim_{\theta\to 0^+}\frac{C}{\theta^2}\int_{\t}^{\t+\theta}
\Big\{\Big[\me~\big|\ms(t)\big|^2\Big]^{\frac{1}{2}}\cdot
\Big[\me~\Big|\int_{\t}^{t}b_{x}(s)\Phi(s)ds\Big|^4\Big]^{\frac{1}{4}}\cdot\\
& &\qquad\qquad\qquad\quad\qquad\qquad
\Big[\me~\Big|\int_{\t}^{t}\big[\Phi(s)^{-1}\sigma_{u}(s)
(v-\bar{u}(s))\big]dW(s)\Big|^4\Big]^{\frac{1}{4}}\Big\}dt\\
&\le&\lim_{\theta\to 0^+}\frac{C}{\theta^2}\int_{\t}^{\t+\theta}
\Big\{\Big[\me~\big|\ms(t)\big|^2\Big]^{\frac{1}{2}}\cdot
\Big[\me~\Big|\int_{\t}^{t}b_{x}(s)\Phi(s)ds\Big|^4\Big]^{\frac{1}{4}}\cdot\\
& &\qquad\qquad\qquad\quad\qquad\qquad
\Big[\me~\Big(\int_{\t}^{t}\big|\Phi(s)^{-1}\sigma_{u}(s)
(v-\bar{u}(s))\big|^{2}ds\Big)^2\Big]^{\frac{1}{4}}\Big\}dt\\
&\le&\lim_{\theta\to 0^+}\frac{C}{\theta^2}\int_{\t}^{\t+\theta}
(t-\t)^{\frac{3}{2}}\Big[\me~\big|\ms(t)\big|^2\Big]^{\frac{1}{2}}dt\\
&=& 0,  \ \qquad a.e.\ \ \t\in[0,T).
\end{eqnarray*}
This implies that
\begin{eqnarray}\label{th31 equ2 martingle}
& &\lim_{\theta\to 0^+}\frac{1}{\theta^2}\int_{\t}^{\t+\theta}\me~
\Big\langle\ms(t)\int_{\t}^{t}b_{x}(s)\Phi(s)ds \cdot\\
& &\qquad\qquad\qquad\qquad
\int_{\t}^{t}\Phi(s)^{-1}\sigma_{u}(s)
(v-\bar{u}(s))dW(s),
v-\bar{u}(t)\Big\rangle dt\nonumber\\
&=& 0  \ \ \qquad a.e.\ \ \t\in[0,T).\nonumber
\end{eqnarray}

Furthermore, since
\begin{eqnarray*}
& &\lim_{\theta\to 0^+}\frac{1}{\theta^2}\int_{\t}^{\t+\theta}\me~
\Big\langle\ms(t)\int_{\t}^{t}\sigma_{x}(s)\Phi(s)dW(s)
\cdot\\
& &\qquad\qquad\qquad\qquad\qquad
\int_{\t}^{t}\Phi(s)^{-1}\sigma_{u}(s)(v-\bar{u}(s))dW(s),
v-\bar{u}(t)\Big\rangle dt\\
&=&\lim_{\theta\to 0^+}\frac{1}{\theta^2}\int_{\t}^{\t+\theta}\me~
\Big\langle \int_{\t}^{t}\sigma_{x}(s)\Phi(s)dW(s)
\int_{\t}^{t}\Phi(s)^{-1}\sigma_{u}(s)
(v-\bar{u}(s)) dW(s),\\
& &\qquad\qquad\qquad\qquad\qquad \qquad\qquad\ \
\ms(t)^{\top}\big(v-\bar{u}(t)\big)
-\ms(\t)^{\top}\big(v-\bar{u}(\t)\big)
\Big\rangle dt\\
& &\quad +\lim_{\theta\to 0^+}\frac{1}{\theta^2}\int_{\t}^{\t+\theta}\me~
\Big\langle\ms(\t)
\int_{\t}^{t}\sigma_{x}(s)\Phi(s)dW(s)\cdot\\
& &\qquad\qquad\qquad\qquad\qquad\
\int_{\t}^{t}\Phi(s)^{-1}\sigma_{u}(s)(v-\bar{u}(s))dW(s),
v-\bar{u}(\t)\Big\rangle dt,
\end{eqnarray*}
and
\begin{eqnarray*}
& &\lim_{\theta\to 0^+}\frac{1}{\theta^2}\Big|\int_{\t}^{\t+\theta}\me~
\Big\langle \int_{\t}^{t}\sigma_{x}(s)\Phi(s)dW(s)
\int_{\t}^{t}\Phi(s)^{-1}\sigma_{u}(s)
(v-\bar{u}(s)) dW(s),\\
& &\qquad\qquad\qquad\qquad\qquad \qquad\qquad\ \
\ms(t)^{\top}\big(v-\bar{u}(t)\big)
-\ms(\t)^{\top}\big(v-\bar{u}(\t)\big)
\Big\rangle dt\Big|\\
&\le& \lim_{\theta\to 0^+}\frac{1}{\theta^2}\int_{\t}^{\t+\theta}
\Big[\me~\Big|\int_{\t}^{t}\big|\sigma_{x}(s)\Phi(s)
\Big|^2ds\Big|^2\Big]^{\frac{1}{4}}\cdot\\
& &\qquad\qquad\qquad \qquad\quad\
\Big[\me~\Big|\int_{\t}^{t}\big|\Phi(s)^{-1}\sigma_{u}(s)
(v-\bar{u}(s))
\big|^2ds\Big|^2\Big]^{\frac{1}{4}}\cdot\\
& &\qquad\qquad\qquad \qquad\qquad \qquad\ \Big[\me~\big|\ms(t)^{\top}\big(v-\bar{u}(t)\big)
-\ms(\t)^{\top}\big(v-\bar{u}(\t)\big)\big|^2\Big]^{\frac{1}{2}}
dt\\
&\le&\lim_{\theta\to 0^+}\frac{C}{\theta^{\frac{1}{2}}}\Big[\int_{\t}^{\t+\theta}
\me~\big|\ms(t)^{\top}\big(v-\bar{u}(t)\big)
-\ms(\t)^{\top}\big(v-\bar{u}(\t)\big)
\big|^2dt\Big]^{\frac{1}{2}}\\
&=&0, \qquad\qquad  a.e.\ \ \t\in[0,T),
\end{eqnarray*}
then, by Lemma \ref{technical lemma}, we have,
\begin{eqnarray}\label{th31 equ3 martingle}
& &\lim_{\theta\to 0^+}\frac{1}{\theta^2}\int_{\t}^{\t+\theta}\me~
\Big\langle\ms(t)\int_{\t}^{t}\sigma_{x}(s)\Phi(s)dW(s)
\cdot\\
& &\qquad\qquad\qquad\qquad\qquad
\int_{\t}^{t}\Phi(s)^{-1}\sigma_{u}(s)(v-\bar{u}(s))dW(s),
v-\bar{u}(t)\Big\rangle dt\nonumber\\
&=&\lim_{\theta\to 0^+}\frac{1}{\theta^2}\int_{\t}^{\t+\theta}\me~
\Big\langle \ms(\t)
\int_{\t}^{t}\sigma_{x}(s)\sigma_{u}(s)(v-\bar{u}(s))ds,
v-\bar{u}(\t)\Big\rangle dt\nonumber\\
&=&\frac{1}{2}\me\inner{\ms(\t)\sigma_{x}(\t)
\sigma_{u}(\t)(v-\bar{u}(\t))}
{v-\bar{u}(\t)},  \ \ \ \ a.e.\ \ \t\in[0,T).\nonumber
\end{eqnarray}

Therefore, by (\ref{limit of S(t)y1(t)part2})--(\ref{th31 equ3 martingle}), we have, for a.e. $\t\in[0,T)$,
\begin{eqnarray}\label{limit of S(t)y1(t)part2final}
\ & &\limsup_{\theta\to 0^+}\frac{1}{\theta^{2}} \me\int_{\t}^{\t+\theta}
\Big\langle \ms(t)\Phi(t)\int_{\t}^{t}
\Phi(s)^{-1}\sigma_{u}(s)(v-\bar{u}(s))dW(s),
v-\bar{u}(t)\Big\rangle dt\\
&=&\frac{1}{2}\partial^{+}_{\t}\big(\ms(\t)^{\top}(v-\bar{u}(\t)),
\sigma_{u}(\t)(v-\bar{u}(\t))\big)\nonumber \\
& & +\frac{1}{2}\me\inner{\ms(\t)\sigma_{x}(\t)\sigma_{u}(\t)(v-\bar{u}(\t))}
{v-\bar{u}(\t)}.\nonumber
\end{eqnarray}

Finally, by (\ref{limit of S(t)y1(t)}), (\ref{limit of S(t)y1(t)part1}) and (\ref{limit of S(t)y1(t)part2final}), we obtain that
\begin{eqnarray*}
0&\ge&\limsup_{\theta\to 0^+}\frac{1}{\theta^{2}}\me\int_{\t}^{\t+\theta}
\inner{\ms(t)y_{1}(t)}{v-\bar{u}(t)}dt\\
&=&\frac{1}{2}\me \inner{\ms(\t)\big(b_{u}(\t)-\sigma_{x}(\t)\sigma_{u}(\t)
\big)(v-\bar{u}(\t))}
{v-\bar{u}(\t)}\\
& &+\frac{1}{2}\partial^{+}_{\t}\big(\ms(\t)^{\top}(v-\bar{u}(\t)),
\sigma_{u}(\t)(v-\bar{u}(\t))\big)\\
& &+\frac{1}{2}\me\inner{\ms(\t)\sigma_{x}(\t)
\sigma_{u}(\t)(v-\bar{u}(\t))}
{v-\bar{u}(\t)}\\
&=&\frac{1}{2}\me \inner{\ms(\t)b_{u}(\t)(v-\bar{u}(\t))}
{v-\bar{u}(\t)}\\
& &+\frac{1}{2}\partial^{+}_{\t}\big(\ms(\t)^{\top}(v-\bar{u}(\t)),
\sigma_{u}(\t)(v-\bar{u}(\t))\big),\quad a.e.\ \t\in[0,T),
\end{eqnarray*}
which gives (\ref{2orderconditionconvex martingleexp}). This completes the proof of Theorem \ref{2orderconditionthconvex pointwise with t}.

\subsection{Proof of Theorem \ref{2orderconditionth}}
Since $W(\cdot)$ is a continuous stochastic process,  $\mf_{t}$ is countably generated for any $t\in[0,T]$. Hence, one can fined a sequence  $\{A_{l}\}_{l=1}^{\infty}\subset \mf_{t}$ such that for any $A\in \mf_{t}$, there exists a subsequence $\{A_{l_{n}}\}_{n=1}^{\infty}\subset \{A_{l}\}_{l=1}^{\infty}$ such that
$\lim_{n\to \infty} P(A \Delta A_{l_{n}})=0$,
where $A \Delta A_{l_{n}}=(A\setminus A_{l_{n}})\bigcup (A_{l_{n}}\setminus A)$. $\mf_{t}$ is also said to be generated by the sequence $\{A_{l}\}_{l=1}^{\infty}$.

Denote by $\{t_{i}\}_{i=1}^{\infty}$ the sequence of rational numbers in $[0,T)$, by $\{v^{k}\}_{k=1}^{\infty}$ a dense subset of $U$. As in \cite{Haussmann76,Tang10}, we choose $\{A_{ij}\}_{j=1}^{\infty}(\subset \mf_{t_i})$ to be a sequence generating $\mf_{t_i}$ (for each $i\in \mn$).
Fix $i,j,k\in \mn$ arbitrarily. For any $\t\in [t_{i},T)$ and $\theta\in(0,T-\t)$, write $E_{\theta}^{i}=[\t,\t+\theta)$,
and define
$$
u_{ij}^{k}(t,\omega)=\left\{
\begin{array}{l}
v^{k}, \qquad\qquad\quad (t,\omega)\in E_{\theta}^{i}\times A_{ij},\\
\bar{u}(t,\omega), \qquad \quad\, (t,\omega)\in \big([0,T]\times \Omega\big) \setminus  \big(E_{\theta}^{i}\times A_{ij}\big).\\
\end{array}\right.
$$
Clearly, $u_{ij}^{k}(\cdot)\in \mmu_{ad}$. Choosing a ``test" function $v(\cdot)$ in (\ref{integraltype 2ordercondition}) as
$$v_{ij}^{k}(t,\omega)=u_{ij}^{k}(t,\omega)-\bar{u}(t,\omega)
=\big(v^{k}-\bar{u}(t,\omega)\big)\chi_{A_{ij}}(\omega)\chi_{E_{\theta}^{i}}(t),\quad (t,\omega)\in [0,T]\times \Omega,$$
we obtain that
\begin{equation}\label{integraltype with  perturbation}
\me\int_{\t}^{\t+\theta}
\inner{\ms(t)y_{ij}^{k}(t)}{v^{k}-\bar{u}(t)}\chi_{A_{ij}}(\omega)dt\le 0,
\end{equation}
where $y_{ij}^{k}(\cdot)$ is the solution to the variational equation (\ref{firstvariequconvex}) with $v(\cdot)$ replaced by $v_{ij}^{k}(\cdot)$.
By (\ref{y1(t)for v(t)}),
\begin{eqnarray}\label{y1(ijk)(t)}
\qquad y_{ij}^{k}(t)&=&\Phi(t)\int_{0}^{t}\Phi(s)^{-1}
\big(b_{u}(s)-\sigma_{x}(s)\sigma_{u}(s)
\big)\big(v^{k}-\bar{u}(s)\big)\chi_{E_{\theta}^{i}}(s)\chi_{A_{ij}}(\omega)ds\\
& & +\Phi(t)\int_{0}^{t}\Phi(s)^{-1}\sigma_{u}(s)\big(v^{k}-\bar{u}(s)\big)
\chi_{E_{\theta}^{i}}(s)\chi_{A_{ij}}(\omega)dW(s).\nonumber
\end{eqnarray}

Substituting (\ref{y1(ijk)(t)}) into (\ref{integraltype with  perturbation}), we have
\begin{eqnarray}\label{limit of S(t)y1(ijk)(t)}
\qquad\qquad 0&\ge&\frac{1}{\theta^2}\me\int_{\t}^{\t+\theta}
\Big\langle \ms(t)\Phi(t)\int_{\t}^{t}\Big[\Phi(s)^{-1}
\big(b_{u}(s)-\sigma_{x}(s)\sigma_{u}(s)
\big)\cdot\\
& &\qquad\qquad\qquad\qquad\qquad\qquad
(v^{k}-\bar{u}(s))\chi_{A_{ij}}(\omega)\Big]ds,
v^{k}-\bar{u}(t)\Big\rangle\chi_{A_{ij}}(\omega)dt\nonumber\\
& &+\frac{1}{\theta^2} \me\int_{\t}^{\t+\theta}
\Big\langle\ms(t)\Phi(t)\int_{\t}^{t}
\Big[\Phi(s)^{-1}\sigma_{u}(s)\cdot\nonumber\\
& &\qquad\qquad\qquad\qquad\qquad
(v^{k}-\bar{u}(s))\chi_{A_{ij}}(\omega)\Big]dW(s),
v^{k}-\bar{u}(t)\Big\rangle\chi_{A_{ij}}(\omega)dt.\nonumber
\end{eqnarray}

By Lemma \ref{technical lemma}, it is immediate that for a.e. $\t\in [t_{i},T)$,
\begin{eqnarray}\label{limit of S(t)y1(ijk)(t)part1}
\qquad& &\lim_{\theta\to 0^+}\frac{1}{\theta^2}\me\int_{\t}^{\t+\theta}
\Big\langle \ms(t)\Phi(t)\int_{\t}^{t}\Big[\Phi(s)^{-1}
\big(b_{u}(s)-\sigma_{x}(s)\sigma_{u}(s)\big)\cdot\nonumber\\
& &\qquad\qquad\qquad\qquad\qquad\qquad
(v^{k}-\bar{u}(s))\chi_{A_{ij}}(\omega)\Big]ds,
v^{k}-\bar{u}(t)\Big\rangle\chi_{A_{ij}}(\omega)dt\\
&=& \frac{1}{2}\me~\Big[\inner{\ms(\t)\big(b_{u}(\t)-\sigma_{x}(\t)\sigma_{u}(\t)
\big)(v^{k}-\bar{u}(\t))}
{v^{k}-\bar{u}(\t)}\chi_{A_{ij}}(\omega)\Big].\nonumber
\end{eqnarray}

Next, we prove that there exists a sequence $\{\theta_{n}\}_{n=1}^{\infty}$ such that $\theta_n\to0^+$ as $n\to\infty$ and
\begin{eqnarray}\label{limit of S(t)y1(ijk)(t)part2final}
\qquad& &\lim_{n\to \infty}\frac{1}{\theta_{n}^{2}} \me\int_{\t}^{\t+\theta_{n}}
\Big\langle \ms(t)\Phi(t)\int_{\t}^{t}
\Big[\Phi(s)^{-1}\sigma_{u}(s)\cdot\\
& &\qquad\qquad\qquad\qquad\quad
(v^{k}-\bar{u}(s))\chi_{A_{ij}}(\omega)\Big]dW(s),
v^{k}-\bar{u}(t)\Big\rangle\chi_{A_{ij}}(\omega)dt\nonumber\\
&=& \frac{1}{2} \me~\Big[\inner{\nabla \ms(\t)\sigma_{u}(\t)
(v^{k}-\bar{u}(\t))}
{v^{k}-\bar{u}(\t)}\chi_{A_{ij}}(\omega)\Big]\nonumber\\
& & -\frac{1}{2} \me~\Big[\inner{\ms(\t)\sigma_{u}(\t)
(v^{k}-\bar{u}(\t))}
{ \nabla\bar{u}(\t)}\chi_{A_{ij}}(\omega)\Big]\nonumber\\
& & +\frac{1}{2}\me~\Big[\inner{\ms(\t)\sigma_{x}(\t)\sigma_{u}(\t)(v^{k}-\bar{u}(\t))}
{v^{k}-\bar{u}(\t)}\chi_{A_{ij}}(\omega)\Big],\  \ a.e. \ \t\in [t_{i},T).\nonumber
\end{eqnarray}

By (\ref{Phi}),
\begin{eqnarray}\label{limit of S(t)y1(ijk)(t)part2}
\qquad& &\frac{1}{\theta^2}\int_{\t}^{\t+\theta} \me~\Big\{
\Big\langle\ms(t)\Phi(t)\int_{\t}^{t}
\Big[\Phi(s)^{-1}\sigma_{u}(s)\cdot\\
& &\qquad\qquad\qquad\qquad\
(v^{k}-\bar{u}(s))\chi_{A_{ij}}(\omega)\Big]dW(s),
v^{k}-\bar{u}(t) \Big\rangle\chi_{A_{ij}}(\omega)\Big\}dt\nonumber\\
&=&\frac{1}{\theta^2}\int_{\t}^{\t+\theta}\me~\Big\{
\Big\langle\ms(t)\Phi(\t)\int_{\t}^{t}
\Big[\Phi(s)^{-1}\sigma_{u}(s)\cdot\nonumber\\
& &\qquad\qquad\qquad\qquad\
(v^{k}-\bar{u}(s))\chi_{A_{ij}}(\omega)\Big]dW(s),
v^{k}-\bar{u}(t) \Big\rangle\chi_{A_{ij}}(\omega)\Big\}dt\nonumber\\
& &+\frac{1}{\theta^2}\int_{\t}^{\t+\theta}\me~\Big\{
\Big\langle\ms(t)\int_{\t}^{t}b_{x}(s)\Phi(s)ds
\int_{\t}^{t}\Big[\Phi(s)^{-1}\sigma_{u}(s)\cdot\nonumber\\
& &\qquad\qquad\qquad\qquad\
(v^{k}-\bar{u}(s))\chi_{A_{ij}}(\omega)\Big]dW(s),
v^{k}-\bar{u}(t)\Big\rangle\chi_{A_{ij}}(\omega)\Big\}dt\nonumber\\
& &+\frac{1}{\theta^2}\int_{\t}^{\t+\theta}\me~\Big\{
\Big\langle\ms(t)\int_{\t}^{t}\sigma_{x}(s)\Phi(s)dW(s)
\int_{\t}^{t}\Big[\Phi(s)^{-1}\sigma_{u}(s)\cdot\nonumber\\
& &\qquad\qquad\qquad\qquad\
(v^{k}-\bar{u}(s))\chi_{A_{ij}}(\omega)\Big]dW(s),
v^{k}-\bar{u}(t)\Big\rangle\chi_{A_{ij}}(\omega)\Big\}dt.\nonumber
\end{eqnarray}
Therefore, we can divide the computation for the left hand side of (\ref{limit of S(t)y1(ijk)(t)part2final}) into three parts.

Similar to respectively (\ref{th31 equ2 martingle}) and (\ref{th31 equ3 martingle}) (in the proof of Theorem \ref{2orderconditionthconvex pointwise with t}), we get that
\begin{eqnarray}\label{limit of S(t)y1(ijk)(t)equ4}
& &\lim_{\theta\to 0^+}\frac{1}{\theta^2}\int_{\t}^{\t+\theta}\me~\Big\{
\Big\langle\ms(t)\int_{\t}^{t}b_{x}(s)\Phi(s)ds \int_{\t}^{t}\Big[\Phi(s)^{-1}\sigma_{u}(s)\cdot\\
& &\qquad\qquad\qquad\qquad\qquad
(v^{k}-\bar{u}(s))\chi_{A_{ij}}(\omega)\Big]dW(s),
v^{k}-\bar{u}(t)\Big\rangle\chi_{A_{ij}}(\omega)\Big\}dt\nonumber\\
&=& 0,  \ \ \ \ a.e.\ \ \t\in[t_{i},T),\nonumber
\end{eqnarray}
and
\begin{eqnarray}\label{limit of S(t)y1(ijk)(t)equ5}
& &\lim_{\theta\to 0^+}\frac{1}{\theta^2}\int_{\t}^{\t+\theta}\me~\Big\{
\Big\langle\ms(t)\int_{\t}^{t}\sigma_{x}(s)\Phi(s)dW(s)
\int_{\t}^{t}\Big[\Phi(s)^{-1}\sigma_{u}(s)\cdot\\
& &\qquad\qquad\qquad\qquad\qquad
(v^{k}-\bar{u}(s))\chi_{A_{ij}}(\omega)\Big]dW(s),
v^{k}-\bar{u}(t)\Big\rangle\chi_{A_{ij}}(\omega)\Big\}dt\nonumber\\
&=&\frac{1}{2}\me~\Big[\inner{\ms(\t)\sigma_{x}(\t)
\sigma_{u}(\t)(v^{k}-\bar{u}(\t)}
{v^{k}-\bar{u}(\t)}\chi_{A_{ij}}(\omega)\Big], \ \  a.e.\ \t\in[t_{i},T).\nonumber
\end{eqnarray}

It remains to  prove that there exists a sequence $\{\theta_{n}\}_{n=1}^{\infty}$ such that $\theta_n\to0^+$ as $n\to\infty$ and
\begin{eqnarray}\label{limit of S(t)y1(ijk)(t)equ3}
& &\lim_{n\to\infty}\frac{1}{\theta_n^2}\int_{\t}^{\t+\theta_n}\me~\Big\{
\Big\langle\ms(t)\Phi(\t)\int_{\t}^{t}
\Big[\Phi(s)^{-1}\sigma_{u}(s)\cdot\\
& &\qquad\qquad\qquad\qquad\
(v^{k}-\bar{u}(s))\chi_{A_{ij}}(\omega)\Big]dW(s),
v^{k}-\bar{u}(t) \Big\rangle\chi_{A_{ij}}(\omega)\Big\}dt\nonumber\\
&=& \frac{1}{2} \me~\Big[\inner{\nabla\ms(\t)\sigma_{u}(\t)
(v^{k}-\bar{u}(\t))}{v^{k}-\bar{u}(\t)}\chi_{A_{ij}}(\omega)\Big]\nonumber\\
& &\qquad\ -\frac{1}{2} \me~\Big[\inner{\ms(\t)\sigma_{u}(\t)
(v^{k}-\bar{u}(\t))}
{\nabla\bar{u}(\t)}\chi_{A_{ij}}(\omega)\Big],\  \  a.e.\ \t\in[t_{i},T).\nonumber
\end{eqnarray}

By the boundness of $U$ and the regularity assumption {\em (C3)}, it holds that
$$\ms(\cdot)^{\top}(v^{k}-\bar{u}(\cdot))\in\ml^{1,2}_{\mmf}(\mrn)\cap L^{\infty}([0,T]\times\Omega;\mr^{n}),$$
Then, by the Clark-Ocone formula, for a.e. $t\in [0,T]$,
\begin{eqnarray}\label{expu(t)S(t)}
\ms(t)^{\top}(v^{k}-\bar{u}(t))
&=&\me~\Big[\ms(t)^{\top}(v^{k}-\bar{u}(t))\Big]\nonumber\\
& &+\int_{0}^{t}\me~\Big[\dd_{s}
\big(\ms(t)^{\top}(v^{k}-\bar{u}(t))\big)\ \Big|\ \mf_{s}\Big]dW(s).
\end{eqnarray}
Substituting (\ref{expu(t)S(t)}) into the first term of the right hand of (\ref{limit of S(t)y1(ijk)(t)part2}), we obtain that
\begin{eqnarray}\label{th31 equ1}
& &\quad\frac{1}{\theta^2}\int_{\t}^{\t+\theta}\me~\Big\{
\Big\langle\ms(t)\Phi(\t)\int_{\t}^{t}
\Big[\Phi(s)^{-1}\sigma_{u}(s)\\
& & \qquad\qquad\qquad\qquad\qquad\
(v^{k}-\bar{u}(s))\chi_{A_{ij}}(\omega)\Big]dW(s),
v^{k}-\bar{u}(t)\Big\rangle\chi_{A_{ij}}(\omega)\Big\}dt\nonumber\\
&=&\frac{1}{\theta^2}\int_{\t}^{\t+\theta}
\me~\Big\{\Big\langle\int_{\t}^{t}\Phi(\t)\Phi(s)^{-1}\sigma_{u}(s)
(v^{k}-\bar{u}(s))\chi_{A_{ij}}(\omega)dW(s),\nonumber\\
& & \qquad\qquad\qquad\qquad\qquad\qquad\qquad\qquad\qquad
\me~\big[\ms(t)^{\top}(v^{k}-\bar{u}(t))\big]
\Big\rangle\chi_{A_{ij}}(\omega)\Big\}dt\nonumber\\
& & \qquad+\frac{1}{\theta^2}\int_{\t}^{\t+\theta}
\me~\Big\{\Big\langle \int_{\t}^{t}\Phi(\t)\Phi(s)^{-1}\sigma_{u}(s)
(v^{k}-\bar{u}(s))\chi_{A_{ij}}(\omega)dW(s),\nonumber\\
& & \qquad\qquad\qquad\qquad\quad\ \
\int_{0}^{t}\me~\Big[\dd_{s}
\big(\ms(t)^{\top}(v^{k}-\bar{u}(t))\big)\ \Big|\ \mf_{s}\Big]dW(s)
\Big\rangle\chi_{A_{ij}}(\omega)\Big\}dt\nonumber\\
&=&\frac{1}{\theta^2}\int_{\t}^{\t+\theta}\int_{\t}^{t}
\me~\Big[\Big\langle\Phi(\t)\Phi(s)^{-1}\sigma_{u}(s)
(v^{k}-\bar{u}(s)),\nonumber\\
& & \qquad\qquad\qquad\qquad\qquad\qquad\qquad\qquad
\dd_{s}\big(\ms(t)^{\top}(v^{k}-\bar{u}(t))\big)
\Big\rangle\chi_{A_{ij}}(\omega)\Big]dsdt.\nonumber
\end{eqnarray}
The last equality in (\ref{th31 equ1}) follows from the fact that~$A_{ij}\in\mf_{t_{i}}\subset\mf_{\t}$ and
\begin{eqnarray*}
& &\me~\Big\{\Big\langle \int_{\t}^{t}\Phi(\t)\Phi(s)^{-1}\sigma_{u}(s)
(v^{k}-\bar{u}(s))\chi_{A_{ij}}(\omega)dW(s),\\
& & \qquad\qquad\qquad\qquad\qquad
\int_{0}^{t}\me~\Big[
\dd_{s}\big(\ms(t)^{\top}(v^{k}-\bar{u}(t))\big)\ \Big|\ \mf_{s}\Big]dW(s)
\Big\rangle\chi_{A_{ij}}(\omega)\Big\}\\
&=&\me~\Big\{\chi_{A_{ij}}(\omega)\me~\Big(\Big\langle
\int_{\t}^{t}\Phi(\t)\Phi(s)^{-1}\sigma_{u}(s)
(v^{k}-\bar{u}(s))\chi_{A_{ij}}(\omega)dW(s),\\
& & \qquad\qquad\qquad\qquad\ \
\int_{\t}^{t}\me~\Big[\dd_{s}\big(\ms(t)^{\top}(v^{k}-\bar{u}(t))\big)
\ \Big|\ \mf_{s}\Big]dW(s)\Big\rangle\ \Big|\ \mf_{\t}\Big)\Big\}\\
&=&\me~\Big\{\chi_{A_{ij}}(\omega)
\me~\Big(\int_{\t}^{t}\Big\langle
\Phi(\t)\Phi(s)^{-1}\sigma_{u}(s)
(v^{k}-\bar{u}(s))\chi_{A_{ij}}(\omega),\\
& & \qquad\qquad\qquad\qquad\qquad\qquad
\me~\Big[\dd_{s}\big(\ms(t)^{\top}(v^{k}-\bar{u}(t))\big)\ \Big|\ \mf_{s}\Big]
\Big\rangle ds\ \Big|\ \mf_{\t}\Big)\Big\}\\
&=&\me\int_{\t}^{t}
\Big\langle \Phi(\t)\Phi(s)^{-1}\sigma_{u}(s)
(v^{k}-\bar{u}(s))\chi_{A_{ij}}(\omega),\\
& & \qquad\qquad\qquad\qquad\qquad\qquad
\me~\Big[\dd_{s}\big(\ms(t)^{\top}(v^{k}-\bar{u}(t))\big)\ \Big|\ \mf_{s}\Big]
\Big\rangle\chi_{A_{ij}}(\omega)ds\\
&=&\int_{\t}^{t}\me~\Big[
\Big\langle \Phi(\t)\Phi(s)^{-1}\sigma_{u}(s)
(v^{k}-\bar{u}(s)), \dd_{s}\big(\ms(t)^{\top}(v^{k}-\bar{u}(t))\big)
\Big\rangle\chi_{A_{ij}}(\omega)\Big]ds.
\end{eqnarray*}

Note that
$$\dd_{s}\big(\ms(t)^{\top}(v^{k}-\bar{u}(t))\big)
=\dd_{s} \ms(t)^{\top}(v^{k}-\bar{u}(t))
-\ms(t)^{\top}\dd_{s}\bar{u}(t).$$
We have,
\begin{eqnarray}\label{th31 equ1add}
\quad & &\frac{1}{\theta^2}\int_{\t}^{\t+\theta}\int_{\t}^{t}
\me~\Big[\Big\langle\Phi(\t)\Phi(s)^{-1}\sigma_{u}(s)
(v^{k}-\bar{u}(s)), \\
& & \qquad\qquad\qquad\qquad\qquad\qquad\qquad\qquad\ \ \
\dd_{s}\big(\ms(t)^{\top}(v^{k}-\bar{u}(t))\big)
\Big\rangle\chi_{A_{ij}}(\omega)\Big]dsdt\nonumber\\
&=&\frac{1}{\theta^2}\int_{\t}^{\t+\theta}\int_{\t}^{t}
\me~\Big[\Big\langle\Phi(\t)\Phi(s)^{-1}\sigma_{u}(s)
(v^{k}-\bar{u}(s)), \nonumber\\
& & \qquad\qquad\qquad\qquad\qquad\qquad\qquad\qquad\ \ \ \ \
\dd_{s} \ms(t)^{\top}(v^{k}-\bar{u}(t))
\Big\rangle\chi_{A_{ij}}(\omega)\Big]dsdt\nonumber\\
& &-\frac{1}{\theta^2}\int_{\t}^{\t+\theta}\int_{\t}^{t}
\me~\Big[\inner{\Phi(\t)\Phi(s)^{-1}\sigma_{u}(s)
(v^{k}-\bar{u}(s))}{\ms(t)^{\top}\dd_{s}\bar{u}(t)
}\chi_{A_{ij}}(\omega)\Big]dsdt.\nonumber
\end{eqnarray}

For the first part in the right hand side of (\ref{th31 equ1add}),
\begin{eqnarray}\label{th31 equ2}
& &\frac{1}{\theta^2}\int_{\t}^{\t+\theta}\int_{\t}^{t}
\me~\Big[\Big\langle \Phi(\t)\Phi(s)^{-1}\sigma_{u}(s)
(v^{k}-\bar{u}(s)), \\
& & \qquad\qquad\qquad\qquad\qquad\qquad\qquad\qquad
\dd_{s} \ms(t)^{\top}(v^{k}-\bar{u}(t))\Big\rangle
\chi_{A_{ij}}(\omega)\Big]dsdt\nonumber\\
&=&\frac{1}{\theta^2}\int_{\t}^{\t+\theta}\int_{\t}^{t}
\me~\Big[\Big\langle \Phi(\t)\Phi(s)^{-1}\sigma_{u}(s)
(v^{k}-\bar{u}(s)),\nonumber\\
& & \qquad\qquad\qquad\qquad\qquad\quad
\big(\dd_{s} \ms(t)-\nabla\ms(s)\big)^{\top}
(v^{k}-\bar{u}(t))\Big\rangle\chi_{A_{ij}}(\omega)\Big]dsdt\nonumber\\
& & +\frac{1}{\theta^2}\int_{\t}^{\t+\theta}\int_{\t}^{t}
\me~\Big[\Big\langle\Phi(\t)\Phi(s)^{-1}\sigma_{u}(s)
(v^{k}-\bar{u}(s)),\nonumber\\
& & \qquad\qquad\qquad\qquad\qquad\qquad\qquad\quad
\nabla\ms(s)^{\top}(v^{k}-\bar{u}(t))\Big\rangle
\chi_{A_{ij}}(\omega)\Big]dsdt.\nonumber
\end{eqnarray}
Since
\begin{eqnarray*}
& &\Big|\frac{1}{\theta^2}\int_{\t}^{\t+\theta}\int_{\t}^{t}
\me\Big[\Big\langle\Phi(\t)\Phi(s)^{-1}\sigma_{u}(s)
(v^{k}-\bar{u}(s)),\\
& & \qquad\qquad\qquad\qquad\qquad
\big(\dd_{s} \ms(t)-\nabla\ms(s)\big)^{\top}
(v^{k}-\bar{u}(t))\Big\rangle\chi_{A_{ij}}(\omega)\Big]dsdt\Big|\\
&\le&\frac{C}{\theta^2}\int_{\t}^{\t+\theta}\int_{\t}^{t}\me~
\Big[\big|\Phi(\t)\Phi(s)^{-1}\sigma_{u}(s)
(v^{k}-\bar{u}(s))\big|\cdot
\big|\dd_{s} \ms(t)-\nabla\ms(s)\big|\Big]dsdt\\
&\le&\frac{C}{\theta}
\Big[\me \Big(\sup_{s\in [\t,T]}|\Phi(\t)\Phi(s)^{-1}|^2\Big)\Big]^{\frac{1}{2}}\cdot
\Big[\me\int_{\t}^{\t+\theta}\int_{\t}^{t}
\Big|\dd_{s} \ms(t)-\nabla\ms(s)\Big|^2dsdt\Big]^{\frac{1}{2}}\\
&\le&\frac{C}{\theta}\Big[\me\int_{\t}^{\t+\theta}\int_{\t}^{t}
\Big|\dd_{s} \ms(t)-\nabla\ms(s)\Big|^2dsdt\Big]^{\frac{1}{2}},
\end{eqnarray*}
by Lemma \ref{lemma for malliavin deriv}, there exists a sequence $\{\theta_{n}\}_{n=1}^{\infty}$ such that $\theta_n\to0^+$ as $n\to\infty$ and
\begin{eqnarray}\label{th31 equ3}
& &\lim_{n\to \infty}\frac{1}{\theta_{n}^{2}}\int_{\t}^{\t+\theta_{n}}\int_{\t}^{t}
\me\Big[\Big\langle\Phi(\t)\Phi(s)^{-1}\sigma_{u}(s)
(v^{k}-\bar{u}(s)),\\
& & \qquad\qquad\qquad\qquad\qquad
\big(\dd_{s}\ms(t)-\nabla\ms(s)\big)
^{\top}(v^{k}-\bar{u}(t))\Big\rangle\chi_{A_{ij}}(\omega)\Big]dsdt\nonumber\\
&=&0\ \qquad a.e. \t\in[0,T).\nonumber
\end{eqnarray}

For the second part in the right hand side of (\ref{th31 equ2}), by Lemma \ref{technical lemma} it follows that
\begin{eqnarray}\label{th31 equ5}
& & \lim_{\theta\to 0^+}\frac{1}{\theta^2}\int_{\t}^{\t+\theta}\int_{\t}^{t}
\me~\Big[\Big\langle\Phi(\t)\Phi(s)^{-1}\sigma_{u}(s)
(v^{k}-\bar{u}(s)),\\
& & \qquad\qquad\qquad\qquad\qquad\qquad\qquad
\nabla\ms(s)^{\top}(v^{k}-\bar{u}(t))\Big\rangle
\chi_{A_{ij}}(\omega)\Big]dsdt\nonumber\\
&=&\frac{1}{2}\me~\Big[\inner{\nabla\ms(\t)\sigma_{u}(\t)(v^{k}-\bar{u}(\t))}
{v^{k}-\bar{u}(\t)}\chi_{A_{ij}}(\omega)\Big],\ \ a.e. \ \t\in [t_{i},T).\nonumber
\end{eqnarray}

Therefore, by (\ref{th31 equ2})--(\ref{th31 equ5}), we conclude that
\begin{eqnarray}\label{th31 equ6}
& &\lim_{n\to \infty}\frac{1}{\theta_{n}^{2}}\int_{\t}^{\t+\theta_{n}}\int_{\t}^{t}
\me~\Big[\Big\langle\Phi(\t)\Phi(s)^{-1}\sigma_{u}(s)
(v^{k}-\bar{u}(s)),\\
& & \qquad\qquad\qquad\qquad\qquad\qquad\qquad
\dd_{s} \ms(t)^{\top}(v^{k}-\bar{u}(t))\Big\rangle
\chi_{A_{ij}}(\omega)\Big]dsdt\nonumber\\
&=&\frac{1}{2}\me~\Big[\inner{\nabla\ms(\t)\sigma_{u}(\t)(v^{k}-\bar{u}(\t))}
{v^{k}-\bar{u}(\t)}\chi_{A_{ij}}\Big],\ \ a.e. \ \t\in [t_{i},T).\nonumber
\end{eqnarray}

In a similar way, we can prove that there exists a subsequence $\{\theta_{n_{l}}\}_{l=1}^{\infty}$ of $\{\theta_{n}\}_{n=1}^{\infty}$  such that
\begin{eqnarray*}
& &\lim_{l\to \infty}\frac{1}{\theta_{n_{l}}^{2}}\int_{\t}^{\t+\theta_{n_{l}}}\int_{\t}^{t}
\me~\Big[\inner{\Phi(\t)\Phi(s)^{-1}\sigma_{u}(s)
(v^{k}-\bar{u}(s))}
{\ms(t)^{\top}\dd_{s}\bar{u}(t)}\chi_{A_{ij}}(\omega)\Big]dsdt\nonumber\\
&=& \frac{1}{2} \me~\Big[\inner{\ms(\t)\sigma_{u}(\t)(v^{k}-\bar{u}(\t))}
{\nabla\bar{u}(\t)}\chi_{A_{ij}}(\omega)\Big],\ \ a.e. \ \t\in [t_{i},T).\nonumber\\
\end{eqnarray*}
To simplify the notation, we assume that the above $\{\theta_{n_{l}}\}_{l=1}^{\infty}$ is  $\{\theta_{n}\}_{n=1}^{\infty}$  itself, that is
\begin{eqnarray}\label{th31 equ7}
& &\lim_{n\to \infty}\frac{1}{\theta_{n}^{2}}\int_{\t}^{\t+\theta_{n}}\int_{\t}^{t}
\me~\Big[\inner{\Phi(\t)\Phi(s)^{-1}\sigma_{u}(s)
(v^{k}-\bar{u}(s))}
{\ms(t)^{\top}\dd_{s}\bar{u}(t)}\chi_{A_{ij}}(\omega)\Big]dsdt\nonumber\\
&=& \frac{1}{2} \me~\Big[\inner{\ms(\t)\sigma_{u}(\t)(v^{k}-\bar{u}(\t))}
{\nabla\bar{u}(\t)}\chi_{A_{ij}}(\omega)\Big],\ \ a.e. \ \t\in [t_{i},T).\nonumber\\
\end{eqnarray}

Combining (\ref{th31 equ1}), (\ref{th31 equ1add}), (\ref{th31 equ6}) and (\ref{th31 equ7}),  we obtain (\ref{limit of S(t)y1(ijk)(t)equ3}). Then, by (\ref{limit of S(t)y1(ijk)(t)part2})--(\ref{limit of S(t)y1(ijk)(t)equ3}),  we obtain (\ref{limit of S(t)y1(ijk)(t)part2final}).

Finally, by (\ref{limit of S(t)y1(ijk)(t)}), (\ref{limit of S(t)y1(ijk)(t)part1}) and (\ref{limit of S(t)y1(ijk)(t)part2final}) we conclude that, for any $i,j,k\in \mn$, there exists a Lebesgue measurable set $E^{k}_{i,j}\subset[t_{i},T)$ with $|E^{k}_{i,j}|=0$ such that
\begin{eqnarray}\label{ae for fixed ijk}
\qquad 0&\ge&\frac{1}{2}\me~\Big[ \inner{\ms(\t)b_{u}(\t)(v^{k}
-\bar{u}(\t))}{v^{k}-\bar{u}(\t)}\chi_{A_{ij}}(\omega)\Big]\\
& &+\frac{1}{2} \me~\Big[ \inner{\nabla \ms(\t)\sigma_{u}(\t)(v^{k}-\bar{u}(\t))}
{v^{k}-\bar{u}(\t)}\chi_{A_{ij}}(\omega)\Big]\nonumber\\
& & -\frac{1}{2} \me~\Big[\inner{\ms(\t)\sigma_{u}(\t)(v^{k}-\bar{u}(\t))}
{\nabla\bar{u}(\t)}\chi_{A_{ij}}(\omega)\Big],\quad \forall \ \t\in [t_{i},T)\setminus E^{k}_{i,j}.\nonumber
\end{eqnarray}
Let $E_{0}=\bigcup_{i,j,k\in \mn} E^{k}_{i,j}$, then $|E_{0}|=0$, and for any $i,j,k\in \mn$,
\begin{eqnarray*}
& &\me~\Big[ \inner{\ms(\t)b_{u}(\t)(v^{k}
-\bar{u}(\t))}{v^{k}-\bar{u}(\t)}\chi_{A_{ij}}(\omega)\Big]\\
& &\quad+ \me~\Big[ \inner{\nabla \ms(\t)\sigma_{u}(\t)(v^{k}-\bar{u}(\t))}
{v^{k}-\bar{u}(\t)}\chi_{A_{ij}}(\omega)\Big]\nonumber\\
& &\quad-\me~\Big[ \inner{\ms(\t)\sigma_{u}(\t)(v^{k}-\bar{u}(\t))}
{\nabla\bar{u}(\t)}\chi_{A_{ij}}(\omega)\Big]\nonumber\\
& &\le 0,
\  \  \forall\ \t\in [t_{i},T)\setminus E_{0}.
\end{eqnarray*}
By the construction of  $\{A_{ij}\}_{i=1}^{\infty}$, the continuity of the filter $\mmf$ and the density of $\{v^{k}\}_{k=1}^{\infty}$, we conclude that
\begin{eqnarray*}
& &\inner{\ms(\t)b_{u}(\t)(v
-\bar{u}(\t))}{v-\bar{u}(\t)}\\
& &\quad+ \inner{\nabla \ms(\t)\sigma_{u}(\t)(v-\bar{u}(\t))}
{v-\bar{u}(\t)}\nonumber\\
& &\quad -\inner{\ms(\t)\sigma_{u}(\t)(v-\bar{u}(\t))}
{\nabla\bar{u}(\t)}\nonumber\\
& &\le 0,\ \ a.s.,
\qquad   \forall\ (\t,v)\in ([0,T]\setminus E_{0})\times U.
\end{eqnarray*}
This completes the proof of Theorem \ref{2orderconditionth}.

\end{document}